\documentclass[a4paper, 11pt]{article}
\usepackage[utf8]{inputenc}	
\usepackage[english]{babel} 

\usepackage{mathtools}
\usepackage{amsmath}
\usepackage{dsfont}
\usepackage{amsthm}
\usepackage{amssymb}
\usepackage{xcolor}
\usepackage{multirow}
\usepackage{bm}
\usepackage{lineno}
\usepackage[pdftex,pdfborder={0 0 0}, 
colorlinks=true, 
linkcolor=blue, 
citecolor=red, 
pagebackref=true, 
]{hyperref}

\usepackage{graphicx}
\usepackage{subcaption}

\usepackage{algorithm,algorithmic}
\usepackage{lineno}

\theoremstyle{plain}
\newtheorem{theorem}{Theorem}

\newtheorem{lemma}[theorem]{Lemma}
\newtheorem{corollary}[theorem]{Corollary}
\newtheorem{remark}[theorem]{Remark}

\newcommand{\R}{\mathbb{R}}

\def\<{{\langle}}
\def\>{{\rangle}}
\def\dd{{\rm d}}

\def\1Lip{1\text{-Lip}}
\def\l({\left(}
\def\r){\right)}

\def\dd{\mathrm{d}}

\DeclareMathOperator*{\argmin}{arg\,min}

\usepackage{hyperref}
\usepackage{vmargin}
\setmarginsrb{3.5cm}{3cm}{3.5cm}{2cm}{0cm}{0cm}{0cm}{1.5cm}
\allowdisplaybreaks[3]

\begin{document}


\title{Generalized conditional gradient and learning in potential mean field games\footnote{This work was supported by a public grant as part of the
Investissement d'avenir project, reference ANR-11-LABX-0056-LMH,
LabEx LMH, and by the FIME Lab (Laboratoire de Finance des Marchés de l'Energie), Paris.} }

\author{J.~Frédéric Bonnans\footnote{Inria and Laboratoire des Signaux et Systèmes, CNRS (UMR 8506), CentraleSupélec, Université Paris-Saclay, 91190 Gif-sur-Yvette, France. 
E-mails:
\href{mailto:frederic.bonnans@inria.fr}{frederic.bonnans@inria.fr},
\href{mailto:pierre.lavigne@inria.fr}{pierre.lavigne@inria.fr},
\href{mailto:laurent.pfeiffer@inria.fr}{laurent.pfeiffer@inria.fr}.
}
\and
Pierre Lavigne\textsuperscript{$\dagger$}\footnote{
CMAP UMR 7641, Ecole Polytechnique, route de Saclay, 91128, Palaiseau Cedex, Institut Polytechnique de Paris, France.}
\and
Laurent Pfeiffer\textsuperscript{$\dagger$}
}


\date{\today}

\maketitle

\begin{abstract}
We apply the generalized conditional gradient algorithm to potential mean field games and we show its well-posedeness. It turns out that this method can be interpreted as a learning method called fictitious play.
More precisely, each step of the generalized conditional gradient method amounts to compute the best-response of the representative agent, for a predicted value of the coupling terms of the game.
We show that for the learning sequence $\delta_k = 2/(k+2)$, the potential cost converges in $O(1/k)$, the exploitability and the variables of the problem (distribution, congestion, price, value function and control terms) converge in $O(1/\sqrt{k})$, for specific norms.
\end{abstract}



\section{Introduction}

Mean field games were introduced by J.-M.~Lasry and P.-L.~Lions in \cite{LL06cr1,LL06cr2,LL07mf} and M.~Huang, R.~Malham\'e, and P.~Caines in \cite{HCMieeeAC06}, to study interactions among a large population of players. Mean field games have found various applications such has
epidemic control \cite{doncel2020mean,elie2020contact}, electricity management \cite{alasseur2020extended,couillet2012electrical}, finance and banking \cite{cardaliaguet2016mean,carmona2017mean,carmona2015probabilistic,feron2020price,lachapelle2016efficiency}, social network \cite{bauso2016opinion}, economics \cite{achdou2014partial,gueant2011mean}, crowd motion \cite{lachapelle2011mean}.
In these models, the nature of the interactions can be of two kinds. Interactions through the density $m$ of players, which appear typically in epidemic or crowd motion models, will be modeled in the following by a congestion function denoted $f$.
Interactions through the controls $v$, which rather appear in economics, finance or energy management models, will be modeled by a price function denoted $\phi$.

\paragraph{Framework} In this article, we study the generalized conditional gradient algorithm to solve potential mean field game problems. We consider the continuous and finite time framework formulated in \cite{BHP-schauder}, consisting of a Hamilton-Jacobi-Bellman equation, a Fokker-Planck equation, and other coupling equations.
We show that the generalized conditional gradient method can be interpreted as a learning procedure called fictitious play. This perspective allows us to:
\begin{enumerate}
\item borrow and apply classical tools from the conditional gradient theory and derive, under suitable assumptions, convergence rates for  the potential cost, the different variables generated by the fictitious play algorithm, and the exploitability;
\item show that the notion of exploitability from game theory is equivalent to the notion of primal-dual gap defined (as defined in Section \ref{sec:gcg}).
\end{enumerate}

\paragraph{Potential mean field games} We say that a mean field game has a convex potential formulation if the congestion and price mappings $f$ and $\phi$ derive from convex potentials $F$ and $\Phi$.
In the mean field game literature, potential (or variational) mean field games were first considered in \cite{LL06cr2}.  This class of games has been widely investigated, we refer the reader to  \cite{benamou2017variational,cardaliaguet2015second,cardaliaguet2016first,meszaros2015variational,prosinski2017global} 
for congestion interactions and \cite{BHP-schauder,graber2015existence,graber2020nonlocal,graber2020mean,graber2018variational,graber2020weak} for price interactions. A key interest of potential mean field games is that the mean field game system stands as sufficient first order conditions for the potential control problem. This is of particular interest for numerical resolution: in such a case one expects classical optimization algorithms to be applicable.

\paragraph{Algorithms} 
The numerical resolution of mean field games has been widely studied, see
\cite{achdou2020mean} for a survey. Primal-dual methods  \cite{bonnans2021discrete,briceno2019implementation,briceno2018proximal} fully use the primal-dual structure of the potential problem.
The augmented Lagrangian algorithm \cite{Benamou2015,benamou2017variational,bonnans2021discrete} is a primal method based on successive minimization of the primal variable and gradient ascent step of dual variables.
Other methods have been investigated such as the Sinkhorn algorithm \cite{benamou2019entropy} or the Mirror Descent algorithm \cite{hadikhanloo2017learning,perolat2021scaling}.

Let us emphasize that most of the above references deal with interaction terms depending on the distribution of the states of the agents; few publications are concerned with interactions through the controls (see \cite{kobeissi2020meanfinite,bonnans2021discrete}). 

\paragraph{Generalized conditional gradient} The generalized conditional gradient algorithm is a variant of the conditional gradient algorithm, also called Frank-Wolfe algorithm, first developed in  \cite{frank1956algorithm}. The conditional gradient method is designed to minimize a convex objective function on a convex and compact set. The idea is to linearize the objective function at each iteration $k \in \mathbb{N}$, at a given point $\bar{x}_k$, and to find a minimizer $x_k$ of this linearized problem. Then a new point $\bar{x}_{k+1} = (1- \delta_k) \bar{x}_k + \delta_k x_k$ is computed for some step size $\delta_k \in [0,1]$. As we will see later, the step size $\delta_k$ can be interpreted as a learning rate for games. A classical choice of step size is given by $\delta_k = 2/(k+2)$ (see \cite{dunn1978conditional,jaggi2013revisiting}) which yields the convergence of the objective function in $O(1/k)$. For a recent description of the conditional gradient algorithm, we refer to \cite[Chapter 1]{kerdreux2020accelerating}. In our study we consider the generalized conditional gradient algorithm (first studied in \cite{bredies2009generalized}), which is based on a semi-linearization of the objective function instead of a full linearization. An interesting feature of this method is that most of the existing convergence results obtained for the conditional gradient remain true for the generalized conditional gradient method. We refer to  \cite{rakotomamonjy2015generalized} for a study. We mention that the previous references deal with finite dimensional problems but these algorithms have been also investigated in infinite dimensional setting, see \cite{bredies2009generalized,pieper2019linear,xu2017convergence} respectively for studies in Hilbert, measures and Banach spaces.

\paragraph{Learning and exploitability} Since most models in social science or engineering rely on Nash equilibria, one can wonder whether such equilibria can be reached if all agents follow their personal interests. Learning is thus a central question in game theory \cite{fudenberg1998theory}. Fictitious play is a best response iterative method for solving games, introduced in \cite{brown1951iterative,robinson1951iterative}. The idea is the following: at each step of the algorithm, for a given belief on the strategy of the others, find the best response of the players; then learn by averaging all the best responses found from the beginning of the learning procedure.
An application of the fictitious play to potential games can be found in \cite{monderer1996potential}.
The fictitious play has been investigated in  \cite{cardaliaguet2017learning,elie2019approximate,HADIKHANLOO2019369,perrin2020fictitious}.
The convergence results for learning methods can be of various forms. In potential games, one can study the convergence of the potential cost along a sequence generated by the fictitious play algorithm. 
In general, one can consider the exploitability of the game at each iteration and try to show its convergence to zero. Given a player and a belief on the others behaviors, the exploitability is the expected relative reward that the player can get by choosing a best response.
This notion has recently received a growing attention \cite{cui2021approximately,delarue2021exploration,geist2021concave,perolat2021scaling,perrin2021mean,perrin2020fictitious}. The convergence of the exploitability has been addressed in \cite{perrin2020fictitious} in the context of continuous time learning and discrete mean field games, and a convergence rate is provided.

\paragraph{Link between the generalized conditional gradient and fictitious play}
A key message of this article is that, in the context of continuous potential mean field games, the generalized conditional gradient algorithm can be interpreted as a fictitious play method.
It relies on the following fact: at each step of the method, the problem to be solved (arising from a semi-linearization of the potential problem) coincides with the individual control problem of the agents, for a given belief of the coupling terms.
The update formula $\bar{x}_{k+1} = (1- \delta_k) \bar{x}_k + \delta_k x_k$ corresponds to the learning step in the fictitious play algorithm, where the agents update their belief by averaging the past and the new distributions of states and controls.

This interpretation has already been highlighted in a very recent work \cite{geist2021concave}, for a class of potential mean field games with some discrete structure. To the best of our knowledge, no other contribution in the literature has investigated the conditional gradient method for mean field games and has pointed out this interpretation.
A minor difference between our framework and the one of \cite{geist2021concave} is the linearity of the running cost of the agents, so that they can apply the classical conditional gradient algorithm (and do not need to rely on semi-linearizations of the potential cost). In our PDE setting, we must employ the standard change of variable ``à la Benamou-Brenier" and the perspective function of the running cost to get a convex potential problem. It turns out that in order to get an interpretation of the method as a learning method, the contribution of the perspective function (in the potential cost) must not be linearized, whence the use of the generalized conditional gradient algorithm.



\paragraph{Contributions}


Our contributions concern the well-posedness of the generalized conditional gradient algorithm and its convergence to the solution of the problem. The well-posedness is established with the help of suitable regularity estimates for the Hamilton-Jacobi-Bellman equation and the Fokker-Planck equation.

Similarly to \cite{geist2021concave}, we use the standard convergence results of the conditional gradient method to prove that the potential cost converges at a rate $O(1/k)$ and the exploitability at a rate $O(1/\sqrt{k})$, when $\delta_k= 2/(k+2)$.

In comparison with \cite{geist2021concave}, the main novelty of our work (besides the different analytical framework) is the proof of convergence of all variables of the game: the coupling terms (price and congestion), the distribution of the agents, and their value function, at a rate $O(1/\sqrt{k})$. A key tool for the proof of convergence is a kind of quadratic growth property satisfied by the potential cost, which itself follows from the (assumed) strong convexity of the running cost of the agents.

Let us mention that we also provide convergence rates for the case $\delta_k= 1/(k+1)$ which is more standard in the fictitious play algorithm: $O(\ln(k)/k)$ for the potential cost, $O(\sqrt{\ln(k)/k})$ for the exploitability and the different variables of the game.


\paragraph{Plan of the paper}

In Section \ref{sec:data} we provide our framework, the mean field game system we are interested in, and give our main assumptions.
In Section \ref{sec:stochastic-pb} we study a stochastic individual control problem. We derive the Hamilton-Jacobi-Bellman equation associated with the value function of the control problem, and provide some regularity results.
We link this problem with a partial differential equation (PDE) control problem of a Fokker-Planck equation and show existence of a (regular) optimal policy.
In Section \ref{section:MFG} we explicit the potential problem under study. We derive uniqueness results for the potential and the individual control problem.
In Section \ref{sec:gcg} we recall the generalized conditional gradient algorithm and apply it to our context. We show that the algorithm is well-defined. We define the exploitability and show the equality with the primal-dual gap. At the end of the section we exhibit the link with the fictitious play learning method.
Finally, in Section \ref{sec:convergence}, we provide our convergence results.

\section{Data and main assumptions \label{sec:data}}

\subsection{Notations}

We fix $T >0$ the duration of the game and $d,k \in \mathbb{N}^\star$ two dimensional coefficients.

\paragraph{Sets}

We set $Q =  \mathbb{T}^d \times  [0,T]$. Given a metric space $X$, we denote by $X^\star$ its dual. 
For any $\alpha \in (0,1)$, we denote by $\mathcal{C}^\alpha(Q)$ the set of H{\"o}lder continuous mappings on $Q$ of exponent $\alpha$ and by $\mathcal{C}^{2+\alpha,1+\alpha/2}(Q)$ the set of continuous mappings $u$ with H{\"o}lder continuous derivatives $\partial_t u$, $\nabla u$ and $D_{xx}^2 u$ on $Q$ of exponent $\alpha$. We also denote by $\mathcal{C}^{1+ \alpha,\alpha}(Q;\R^d)$ the set of all $v \in \mathcal{C}^{\alpha}(Q;\R^d)$ with $D_x v \in \mathcal{C}^{\alpha}(Q,\R^{d \times d})$.

 Sobolev spaces are denoted by $W^{n,q}(Q)$, the order of derivation $n$ being possibly non-integral (following the definition in \cite[section II.2]{LSU}). We set
\begin{equation} \nonumber
W^{2,1,q}(Q) = W^{1,q}(Q) \cap L^q(0,T;W^{2,q}(\mathbb{T}^d)), \quad W^{1,0,q}(Q) = L^q(0,T;W^{1,q}(\mathbb{T}^d)).
\end{equation}
We define
\begin{equation} \nonumber
\mathcal{D}_1(\mathbb{T}^d) = \left\{ m \in L^\infty(\mathbb{T}^d), \; m \geq 0 ,\; \int_{\mathbb{T}^d} m(x) \dd x = 1 \right\}.
\end{equation}
We fix a real number $p$ such that $p > d + 2$.

\paragraph{Nemytskii notations}

For any mappings $g\colon Q \times \mathbb{R}^d \to \mathbb{R}^d $ and $u \colon Q \to \mathbb{R}^d$, we define $\bm{g}[u] \colon Q \to \mathbb{R}$,
\begin{equation} \nonumber
\bm{g}[u](x,t) = g(x,t,u(x,t))
\end{equation}
called Nemytskii operator. This notation will mainly be used for the Hamiltonian $H$. Note that $\bm{H}_p$ will denote the Nemytskii operator associated with the partial derivative of $H$ with respect to $p$ (a similar notation will be used for the other partial derivatives).

\paragraph{Data of the problem}
We fix an initial distribution and a terminal cost
\begin{equation} \nonumber
m_0 \in \mathcal{D}_1(\mathbb{T}^d), \qquad g \colon \mathbb{T}^d \rightarrow \mathbb{R}, 
\end{equation}
and four maps: a running cost $L$, a  congestion cost $f$, a vector of  price $\phi$ and an aggregation term $a$, 
\begin{equation} \nonumber
\begin{array}{rl}
L\colon & Q \times \mathbb{R}^d \rightarrow \R, \\
f \colon & Q \times \mathcal{D}_1(\mathbb{T}^d)\to \R, \\
\phi \colon & [0,T] \times \mathbb{R}^k \rightarrow \R^d, \\
a \colon & Q \to \mathbb{R}^{k \times d}.
\end{array}
\end{equation}
We assume that $L$ is strongly convex, more precisely, we assume that there exists a constant $C_0>0$ such that
for any $v, v' \in \mathbb{R}^d$ and for any $(x,t) \in Q$, we have
\begin{equation}  \label{eq:grad_monotony} \tag{A1}
\langle L_v(x,t,v)-  L_v(x,t,v'), v - v' \rangle \geq \frac{1}{C_0} |v - v' |.
\end{equation}
For any $(x,t,p) \in Q \times \mathbb{R}^d$, we define the Hamiltonian $H$,
\begin{equation} \nonumber
H(x,t,p) = \sup_{v \in \mathbb{R}^d} \;  - \langle p, v \rangle - L(x,t,v).
\end{equation}
The strong convexity assumption on $L$ ensures that $H$ takes finite values and is continuously differentiable (more regularity properties on $H$ are collected in Appendix \ref{appendix:reg-H}). We define the perspective function $\tilde{L}\colon  Q \times \mathbb{R} \times \mathbb{R}^d \rightarrow \R $,
\begin{equation} \label{def:perspective-function}
\tilde{L}(x,t,m,w) = 
\begin{cases}
\begin{array}{ll}
m L \left( x,t,
\frac{w}{m}\right), & \text{if $m >0$}, \\
0, & \text{if $m=0$ and $w=0,$} \\
+ \infty, & \text{otherwise.}
\end{array}
\end{cases}
\end{equation}
Note that $\tilde{L}$ is convex and lower semi-continuous with respect to $(m,w)$.
We define $A : L^1(Q;\mathbb{R}^d) \to L^1(0,T;\mathbb{R}^k)$ and $A^\star : L^1(0,T;\mathbb{R}^k) \to L^1(Q;\mathbb{R}^d)$ as follows,
\begin{equation}  \nonumber
A[w](t) =\int_{\mathbb{T}^d} a(x,t) w(x,t) \dd x, \qquad  \quad A^\star[P](x,t) = a^\star(x,t) P(t),
\end{equation}
for any $(x,t) \in Q$.

\subsection{Coupled system and assumptions \label{subsec:coupled_system}} 

The mean field game system under study is the following,
\begin{equation} \label{eq:MFG-gcg} \tag{MFG}
\begin{cases}
\begin{array}{clc}
\text{(i)} &
\begin{cases}
- \partial_t u - \Delta u + \bm{H}[  \nabla u + A^\star P] = \gamma, \\
 u(x,T) =  g(x),
\end{cases} & \begin{array}{r}
(x,t) \in Q, \\
x\in \mathbb{T}^d,
\end{array}
\\[1.5em]
\text{(ii)} & v  = -  \bm{H}_p[  \nabla u +  A^\star P], & (x,t) \in Q,
\\[1em]
\text{(iii)} & \begin{cases}
\partial_t m - \Delta m + \nabla \cdot (v m) = 0,  \\
m(0,x) = m_0(x), 
\end{cases} &  \begin{array}{r}
(x,t) \in Q, \\
x\in \mathbb{T}^d,
\end{array} \\[1.5em]
\text{(iv)} & \gamma(x,t) = f(x,t,m(t)), & (x,t) \in Q, \\[1em]
\text{(v)} & P(t) = \bm{\phi}[A[v m]](t), & t \in [0,T],
\end{array}
\end{cases}
\end{equation}
where the unknown is $(m,v,u,\gamma,P)$ with $m(x,t) \in \mathbb{R}$, $v(x,t) \in \mathbb{R}^d$, $u(x,t) \in \mathbb{R}$, $\gamma(x,t) \in \mathbb{R}$, and $P(t) \in \mathbb{R}^k$, for any $(x,t) \in Q$. The equation (\ref{eq:MFG-gcg},i) is a Hamilton-Jacobi-Bellman equation and describes the evolution of the value function as time goes backward. Equation (\ref{eq:MFG-gcg},ii) defines the optimal control $v$, which is given by the gradient $H_p$ of the Hamiltonian. Equation (\ref{eq:MFG-gcg},iii) is a Fokker-Planck equation, describing the evolution of the state distribution of the agents. Equation (\ref{eq:MFG-gcg},iv) defines the congestion $\gamma$ and equation (\ref{eq:MFG-gcg},v) the price $P$.

\paragraph{Regularity assumptions}

We assume that $\bm{L}_v$ is differentiable with respect to $x$ and $v$ and that $a$ is differentiable with respect to $x$.
All along the article, we make use of the following assumptions. 

\paragraph{Growth assumptions}

There exists $C_0>0$ such that for all $(x,t) \in Q$, $y \in \mathbb{T}^d$, $v \in \R^d$, $z \in \R^k$, and $m \in \mathcal{D}_1(\mathbb{T}^d)$,
\begin{align*}
&  L(x,t,v) \leq C_0 |v|^2 + C_0, \label{ass_L_quad_growth2} \tag{A2} \\
&  |L(x,t,v)-L(y,t,v)| \leq C_0 |x-y| (1 + |v|^2), \label{ass_L_Lipschitz} \tag{A3} \\
&  |\phi(t,z)| \leq C_0, \label{ass_psi_bounded} \tag{A4} \\
&  |f(x,t,m)| \leq C_0. \label{ass_f_bounded} \tag{A5}
\end{align*}

\paragraph{H{\"o}lder and Lipschitz continuity assumptions}

For all $R>0$, there exists $\alpha_0 \in (0,1)$ such that
\begin{align*}
\begin{cases}
\begin{array}{rl}
L \in & \mathcal{C}^{\alpha_0}(B_R), \\
L_v \in & \mathcal{C}^{\alpha_0}(B_R, \R^d), \\
L_{vx} \in & \mathcal{C}^{\alpha_0}(B_R, \R^{d \times d}), \\
L_{vv} \in &  \mathcal{C}^{\alpha_0}(B_R, \R^{d \times d}),
\end{array}
\end{cases}
\quad
\begin{cases}
\begin{array}{rl}
\phi \in & \mathcal{C}^{\alpha_0}(B_R',\R^d), \\
a \in & \mathcal{C}^{\alpha_0}(Q,\R^{k \times d}), \\
  D_x a \in & \mathcal{C}^{\alpha_0}(Q,\R^{k \times d \times d}),
\end{array}
\end{cases} \qquad \label{ass_holder} \tag{A6}
\end{align*}
where $B_R= Q \times B(\R^d,R)$ and $B_R'= [0,T] \times B(\R^k,R)$. There exists $\alpha_0 \in (0,1)$ and $C_0>0$ such that
\begin{align*}
| f(x_2,t_2,m_2) - f(x_1,t_1,m_1) |
 \leq C_0 \left( |x_2 - x_1| + |t_2 - t_1|^{\alpha_0} + \| m_2 - m_1 \|_{L^{2}(\mathbb{T}^d)} \right), \label{ass_hold_f} \tag{A7} 
\end{align*}
for all $(x_1,t_1)$ and $(x_2,t_2) \in Q$ and for all $m_1$ and $m_2 \in \mathcal{D}_1(\mathbb{T}^d)$. 
We further assume that $\phi$ is Lipschitz continuous with respect to its second variable,
\begin{align}
| \phi(t,z_2) - \phi(t,z_1)| & \leq C_0 | z_2 - z_1 |, \label{ass_Lip_phi} \tag{A8} 
\end{align}
for all $(x,t) \in Q$, for all $z_1$ and $z_2 \in \mathbb{R}^k$.

\begin{remark}
Note that compared to the framework of \cite{BHP-schauder} the Assumptions \eqref{ass_psi_bounded} and \eqref{ass_hold_f} are strengthened. Indeed, we require here more regularity: on $f$ with respect to its third variable; on $\phi$ with respect to its second variable.
\end{remark}

\paragraph{Boundary conditions and convention on constants}

We assume that there exists $ \varepsilon_0>0$ such that
 $m_0(x) \geq \varepsilon_0$ for any $x\in \mathbb{T}^d$. There exists $\alpha_0 \in (0,1)$ such that 
\begin{equation}
m_0 \in \mathcal{C}^{2+ \alpha_0}(\mathbb{T}^d), \quad g \in \mathcal{C}^{2 +\alpha_0}(\mathbb{T}^d).
\label{ass_init_cond} \tag{A9}
\end{equation}

All along the article, we make use of two generic constants $C>0$ and $\alpha \in (0,1)$. The value of $C$ may increase from an inequality to the next one; the value of $\alpha$ may decrease. The constants depend on the data of the problem introduced above.

\subsection{Potentials}

\paragraph{Congestion}

We assume that $f$ is monotone, that is to say,
\begin{equation} \nonumber
\int_{\mathbb{T}^d}(f(x,t,m_2) -  f(x,t,m_1))(m_2(x)-m_1(x)) \dd x \geq 0,
\end{equation}
for any $m_1$ and $m_2 \in \mathcal{D}_1(\mathbb{T}^d)$ and for any $t \in [0,T]$.
We assume that $f$ has a primitive, that is, we assume the existence of a map $F \colon [0,T] \times \mathcal{D}_1(\mathbb{T}^d)$ such that
\begin{equation} \label{primitive:F}
F(t,m_2) - F(t,m_1) =\int_0^1 \int_{\mathbb{T}^d} f(x,t, s m_2 + (1-s)m_1)( m_2(x) - m_1(x)) \dd x \dd s.
\end{equation}
The monotonicity assumption implies that
\begin{equation} \nonumber
F(t,m_2) \geq F(t,m_1) + \int_{\mathbb{T}^d} f(x,t,m_1)(m_2(x)-m_1(x)) \dd x. 
\end{equation}
Since this inequality holds for any $m_1 \in \mathcal{D}_1(\mathbb{T}^d)$, $F$ is convex with respect to its second variable as the supremum of affine functions.

\paragraph{Price}

We assume that $\phi$ has a convex potential $\Phi$, that is to say there exists a measurable mapping $\Phi : [0,T] \times \mathbb{R}^k \to \mathbb{R}$, convex with respect to its second variable and such that $\phi(t,z) = \nabla_z \Phi(t,z)$ for any $(t,z) \in  [0,T] \times \mathbb{R}^k$.

\section{Estimates for the individual control problem} \label{sec:stochastic-pb}

In this section we establish regularity results on the  variables $u$, $v$, and $m$, when obtained by solving the equations (\ref{eq:MFG-gcg},i-iii),
for fixed congestion and price.
We investigate the stochastic optimal control problem associated with the HJB equation (\ref{eq:MFG-gcg},i).
In the section we fix $\beta \in (0,1)$ and we consider
\begin{equation} \label{eq:setUbeta}
\mathcal{U}^\beta = \mathcal{C}^{1,\beta}(Q) \times \mathcal{C}^\beta(0,T;\mathbb{R}^k).
\end{equation}
We also fix a pair $(\gamma,P)\in \mathcal{U}^\beta$ and a constant $R>0$ such that
\begin{equation} \label{ineq:assumption-lemma-gamma-P}
\|\gamma\|_{L^\infty(Q)} + \|\nabla \gamma\|_{L^\infty(Q;\mathbb{R}^d)} + \| P \|_{L^\infty(0,T;\mathbb{R}^k)} \leq R.
\end{equation}

\subsection{The individual problem as a stochastic optimal control problem} \label{sec:individual-control-pb}

Let $(B_s)_{s \in [0,T]}$ denote a Brownian motion and let $Y$ be a random variable, independent of $(B_s)_{s \in [0,T]}$, with probability distribution $m_0$.
Let $\mathbb{F}$ denote the filtration generated by the Brownian motion $(B_s)_{s \in [0,T]}$ and the initial random variable $Y$.
We denote by $L_{\mathbb{F}}^2(t,T;\mathbb{R}^d)$ (resp.\@ $L_{\mathbb{F}}^{2,K}(t,T;\mathbb{R}^d)$, for some constant $K>0$) the set of progressively measurable stochastic processes $\nu$ on $[t,T]$ with value in $\mathbb{R}^d$ such that $\mathbb{E} \left[ \int_t^T |\nu_s|^2 \dd s \right] < +\infty$ (resp.\@ $\mathbb{E} \left[ \int_t^T |\nu_s|^2 \dd s \right] \leq K$).
For all $\nu \in L_{\mathbb{F}}^2(t,T;\mathbb{R}^d)$, we denote by $(X^{\nu}_{s})_{s \in [0,T]}$ the solution to the stochastic differential equation 
\begin{equation} \nonumber
\dd X_s = \nu_s \dd s + \sqrt{2} \dd B_s, \quad X_0 = Y.
\end{equation}
We define the individual cost $Z_{\gamma,P} \colon L_{\mathbb{F}}^2(0,T;\mathbb{R}^d) \to \mathbb{R}$,
\begin{equation} \label{pb:individual-cost}
 Z_{\gamma,P}(\nu) =  \mathbb{E} \left[ \int_0^T L(X_s^\nu,s,\nu_s) + \langle A^\star [P](X_s^\nu,s) , \nu_s \rangle + \gamma(X_s^\nu,s) \dd s + g(X_T^\nu) \right].
\end{equation}
We consider the following stochastic individual control problem
\begin{equation} \label{pb:individual-stochastic-control-problem}  \tag{P$_{\gamma,P}$}
\inf_{\nu \in  L_{\mathbb{F}}^2(0,T;\mathbb{R}^d)} Z_{\gamma,P}(\nu).
\end{equation}
 This problem will play an important role in the following, in particular in learning procedures: at each step, a representative player assumes the behavior of the others to be given and solves \eqref{pb:individual-stochastic-control-problem}.

We define the mapping $J_{\gamma,P} \colon Q \times  L_{\mathbb{F}}^2(0,T;\mathbb{R}^d) \to \mathbb{R}$,
\begin{equation} \nonumber
J_{\gamma,P}(x,t,\nu) =  \mathbb{E} \left[\int_t^T  L(X_s,s,\nu_s) + \langle A^\star [P](X_s,s) , \nu_s \rangle + \gamma(X_s,s) \dd s + g(X_T) \right],
\end{equation}
where $(X_{s})_{s \in [t,T]}$ is the solution to
\begin{equation*}
\dd X_s = \nu_s \dd s + \sqrt{2} \dd B_s, \quad X_t= x.
\end{equation*}
We define by $\mathsf{u}[\gamma,P] \colon Q \to \mathbb{R}$ the value function associated with the individual control problem \eqref{pb:individual-stochastic-control-problem},
\begin{equation} \label{eq:value function}
\mathsf{u}[\gamma,P](x,t) = \inf_{\nu \in L_{\mathbb{F}}^2(t,T;\mathbb{R}^d)} J_{\gamma,P}(x,t,\nu).
\end{equation}

\begin{lemma} \label{lem:u-L-infty}
Let $u= \mathsf{u}[\gamma,P]$.
There exists a constant $C>0$, only depending on $R$, such that
\begin{equation*}
u(x,t) = \inf_{\nu \in L_{\mathbb{F}}^{2,C}(t,T;\mathbb{R}^d)} J_{\gamma,P}(x,t,\nu)
\end{equation*}
for a.e.\@ $(x,t) \in Q$, i.e.\@ the optimization set in \eqref{eq:value function} can be restricted to $L_{\mathbb{F}}^{2,C}(t,T;\mathbb{R}^d)$ (the set is defined in the beginning of section \ref{sec:individual-control-pb}).
\end{lemma}

\begin{proof}
We first derive a lower bound of $L$.
By assumption \eqref{ass_holder},
$L(x,t,0)$ and $L_v(x,t,0)$ are bounded. It follows then from the
strong convexity assumption \eqref{eq:grad_monotony}
that there exists a constant $C> 0$ such that
\begin{equation} \label{L_quad_growth1}
  \frac{1}{C} |\nu|^2 - C \leq L(x,t,\nu), \quad
  \text{ for all }(x,t,\nu) \in Q \times \R^d.
\end{equation}
Then, for any $(x,s) \in Q$ and for any $\nu \in \mathbb{R}^d$, we have the following estimates:
\begin{align} \nonumber
L(x,s,\nu) + \langle A^\star [P](x,s) , \nu \rangle & \geq  \frac{1}{C} |\nu|^2 - \|a\|_{L^\infty(Q;\mathbb{R}^{k \times d})} |P(s)| |\nu| - C \\
& \geq \frac{1}{C} (|\nu|^2 - |P(s)|^2 - 1) \geq \frac{1}{C} (|\nu|^2 - 1). \nonumber
\end{align}
Let $t\in [0,T]$, let $\varepsilon \in (0,1)$ and let $\tilde{\nu} \in L_{\mathbb{F}}^2(t,T;\mathbb{R}^d)$ be an $\varepsilon$-optimal process.
Using the bound on $g$ given in Assumption \eqref{ass_init_cond} and using inequality \eqref{ineq:assumption-lemma-gamma-P}, we deduce from the above inequality that
\begin{align}  \nonumber
\mathbb{E} \left[ \int_t^T |\tilde{\nu}_s|^2 \dd s \right]  & \leq C \left( \inf_{\nu \in L_{\mathbb{F}}^2(t,T;\mathbb{R}^d)} J_{\gamma,P}(x,t,\nu) +  \varepsilon + 1 \right)\\
&  \leq C \left(\mathsf{u}[\gamma,P](x,t) +  2\right) \leq C, \nonumber
\end{align}
where the constant $C$ does not depend on $t$ and $\varepsilon$.
Thus any $\varepsilon$-optimal process lies in $L_{\mathbb{F}}^{2,C}(t,T;\mathbb{R}^d)$, which concludes the proof.
\end{proof}
We now consider the Hamilton-Jacobi-Bellman equation
\begin{equation} \label{eq:HJB}
\begin{array}{rlr}
- \partial_t u - \Delta u + \bm{H}[  \nabla u + A^\star P] = & \gamma, & (x,t) \in Q, \\
 u(x,T) = &  g(x), & x\in \mathbb{T}^d.
\end{array}
\end{equation}
By the classical dynamic programming theory, we know that $\mathsf{u}[\gamma,P]$ is the unique viscosity solution to \eqref{eq:HJB}.

\begin{lemma} \label{eq:HJB-unique-solution}
 There exists $\alpha \in (0,1)$, depending on $\gamma$ and $P$, such that $\mathsf{u}[\gamma,P] \in \mathcal{C}^{2+\alpha,1+\alpha/2}(Q)$. In addition there exists a constant $C >0$, only depending on $R$, such that 
 \begin{equation} \nonumber
 \|\mathsf{u}[\gamma,P]\|_{W^{2,1,p}(Q)} + \|\nabla \mathsf{u}[\gamma,P]\|_{W^{2,1,p}(Q)} \leq C.
 \end{equation}
\end{lemma}

\begin{proof}
The proof is given in Appendix \ref{Appendix:HJB}. 
\end{proof}

\subsection{The individual problem as a PDE optimal control problem}

We consider in this subsection an equivalent formulation of \eqref{pb:individual-stochastic-control-problem} as an optimal control problem of the Fokker-Planck equation. To this purpose, we consider the mapping $\mathsf{m} : W^{1,0,\infty}(Q)  \to W^{2,1,p}(Q)$ which associates to any $v \in W^{1,0,\infty}(Q)$ the solution to the Fokker-Planck equation
\begin{equation}
\begin{array}{rlr}
\partial_t m - \Delta m + \nabla \cdot (v m)  =& 0, \quad & (x,t) \in Q, \\
m(x,0)  =& m_0(x), & x \in \mathbb{T}^d.
\end{array}
\end{equation}

\begin{lemma} \label{Lemma:m-mapping}
The mapping $\mathsf{m}$ is well defined. Moreover, for any $v \in W^{1,0,\infty}(Q)$, we have $\mathsf{m}[v](x,t) > 0$, for any $(x,t) \in Q$.
\end{lemma}

\begin{proof}
Direct consequence of Lemma \ref{lemma:maximum-principle}.
\end{proof}
 We define $\mathcal{B}^p = W^{2,1,p}(Q) \times W^{1,0,\infty}(Q)$ (recall that $p > d+2$ is fixed)
and we define
\begin{align} \nonumber
\mathcal{R} &= \left\{(m,v) \in \mathcal{B}^p, \partial_t m - \Delta m + \nabla \cdot (v m)  = 0, \, m(0) = m_0, \,  (x,t) \in Q \right\},\\
\tilde{\mathcal{R}}& =  \left\{(m,w) \in \mathcal{B}^p,\, \partial_t m - \Delta m + \nabla \cdot w = 0, \, m(0) = m_0, \, m(x,t) > 0,  \,(x,t) \in Q \right\}. \nonumber
\end{align}

\begin{lemma} \label{lemma:bijective-chi}
The mapping $\chi \colon \mathcal{R} \to \tilde{\mathcal{R}}$ given by $\chi(m,v) = (m,mv)$  is well-posed and bijective. Its inverse is given by $\chi^{-1}(m,w) = (m,w/m)$. 
\end{lemma}

\begin{proof}
Let $(m,v) \in \mathcal{R}$. We have that $m = \mathsf{m}[v] \in W^{2,1,p}(Q)$, thus $m \in L^\infty(Q)$ and $\nabla m \in L^\infty(Q;\R^d)$, by Lemma \ref{lemma:max_reg_embedding}. It follows that $w\coloneqq \mathsf{m}[v] v \in W^{1,0,\infty}(Q)$.
Moreover, $m> 0$, by Lemma \ref{Lemma:m-mapping}. Therefore $(m,w) \in \tilde{\mathcal{R}}$, that is, $\chi$ is well defined.
Similarly, for any $(m,w) \in \tilde{\mathcal{R}}$, we have that  $w/m \in W^{1,0,\infty}(Q)$ and $\mathsf{m}[w/m] \in W^{2,1,p}(Q)$.
Obviously we have $\chi \circ \chi^{-1} = id$ and $\chi^{-1}  \circ   \chi = id$, which concludes the proof.
\end{proof}

\begin{remark} \label{rem:equivalence}
Let $(m,v) \in \mathcal{R}$ and let $(m,w)= \chi(m,v) \in \tilde{\mathcal{R}}$. Recalling the definition of the perspective function \eqref{def:perspective-function}, we have
\begin{equation}
\int_{Q} \bm{L}[v] m \dd x \dd t = \int_{Q} \tilde{\bm{L}}[m,w] \dd x \dd t. \nonumber
\end{equation}
This fact, together with the existence of a bijection between $\mathcal{R}$ and $\tilde{\mathcal{R}}$, will allow to prove the equivalence of the optimal control problems, introduced later, posed over $\mathcal{R}$ and $\tilde{\mathcal{R}}$.
\end{remark}

We define the individual cost $\mathcal{Z}_{\gamma,P} \colon \mathcal{R} \to \mathbb{R}$,
\begin{equation} \nonumber
\mathcal{Z}_{\gamma,P}(m,v) = \int_{Q} \left( \bm{L}[v] + \gamma \right) m \dd x \dd t + \int_0^T \langle A[m v], P \rangle \dd t + \int_{\mathbb{T}^d}g m(T) \dd x.
\end{equation}
We define the following individual control problem
\begin{equation} \label{pb:individual-control-problem} \tag{$\mathcal{P}_{\gamma,P}$}
\inf_{(m,v) \in \mathcal{R}} \mathcal{Z}_{\gamma,P}(m,v).
\end{equation}
Here the state equation of the agent is a Fokker-Planck equation with controlled drift $v$. 
We define the individual cost $\tilde{\mathcal{Z}}_{\gamma,P} \colon \tilde{\mathcal{R}} \to \mathbb{R}$,
\begin{equation} \nonumber
\tilde{\mathcal{Z}}_{\gamma,P}(m,w) = \int_{Q} \left( \tilde{\bm{L}}[m,w] + \gamma m \right) \dd x \dd t + \int_0^T \langle A[w], P \rangle \dd t + \int_{\mathbb{T}^d}g m(T) \dd x,
\end{equation}
where $\tilde{\bm{L}}$ is the perspective function of $\bm{L}$ (see the definition \eqref{def:perspective-function}),
and the following control problem
\begin{equation} \label{pb:individual-control-problem-w} \tag{$\tilde{\mathcal{P}}_{\gamma,P}$}
\inf_{(m,w) \in \tilde{\mathcal{R}}} \tilde{\mathcal{Z}}_{\gamma,P}(m,w).
\end{equation}

Given $v \in W^{1,0,\infty}(Q)$, we denote $(X^{v}_{s})_{s \in [0,T]}$ the solution to the following stochastic differential equation
\begin{equation}
\dd X_s = v(X_s,s) \dd s + \sqrt{2} \dd B_s, \quad X_0= Y.
\end{equation}
We further consider the associated control $\nu^v_s \in L_{\mathbb{F}}^{2}(0,T;\mathbb{R}^d)$ defined by $\nu^v_s = v(s,X^v_s)$.

\begin{lemma} \label{lemma:equality-of-cost}
For any $v \in W^{1,0,\infty}(Q, \mathbb{R}^d)$, we have
\begin{equation*}
Z_{\gamma,P}(\nu^v)  = \mathcal{Z}_{\gamma,P}(\mathsf{m}[v],v) = \tilde{\mathcal{Z}}_{\gamma,P} \circ \chi (\mathsf{m}[v],v).
\end{equation*}
\end{lemma}

\begin{proof}
It is clear that $\mathcal{Z}_{\gamma,P}(\mathsf{m}[v],v) = \tilde{\mathcal{Z}}_{\gamma,P}\circ \chi (\mathsf{m}[v],v)$, see Remark \ref{rem:equivalence}.
Since $v \in W^{1,0,\infty}(Q, \mathbb{R}^d)$, the process $\nu^v$ lies in  $L_{\mathbb{F}}^{2}(0,T;\mathbb{R}^d)$ and  $Z_{\gamma,P}(\nu^v) < + \infty$.
For any $t \in [0,T]$, $\mathsf{m}[v](\cdot,t)$ is the probability density of the distribution of $X_t^v$. In addition we have by definition that $\nu^v_t = v(t,X^v_t)$, which yields that $Z_{\gamma,P}(\nu^v)  = \mathcal{Z}_{\gamma,P}(\mathsf{m}[v],v)$.
\end{proof}

\begin{lemma} \label{lemma:link-sto-edp} 
Let $u = \mathsf{u}[\gamma,P]$ and let $v = - \bm{H}_p[\nabla u + A^\star P]$.
Let $m= \mathsf{m}[v]$ and let $(m,w)= \chi(m,v)$.
\begin{enumerate}
\item \label{point:vmw-c-alpha} There exists $\alpha \in (0,1)$, depending on $\gamma$ and $P$, such that
\begin{equation} \nonumber
v \in \mathcal{C}^{1+ \alpha,\alpha}(Q;\R^d), \quad
m \in \mathcal{C}^{2+\alpha,1+ \alpha/2}(Q), \quad
w \in \mathcal{C}^{1+ \alpha,\alpha}(Q;\R^d).
\end{equation}
\item \label{point:mvw-W-21p} There exists $C >0$, depending only on $R$, such that
\begin{equation} \nonumber
\| v \|_{W^{1,0,\infty}(Q;\R^d)} \leq C, \quad
\| m \|_{W^{2,1,p}(Q)} \leq C, \quad
\| w \|_{W^{1,0,\infty}(Q;\R^d)} \leq C.
\end{equation}
\item \label{point-verif} The stochastic process $(\nu^v_s)_{s \in [0,T]}$ is the solution to \eqref{pb:individual-stochastic-control-problem}.
\item \label{point-opti-mv-indiv} The pair $(m,v)$ is a solution to \eqref{pb:individual-control-problem} and $(m,w)$ is a solution to \eqref{pb:individual-control-problem-w}.
\end{enumerate}
\end{lemma}

\begin{proof} 
\noindent \textit{Point \ref{point:vmw-c-alpha}.}
We know that $H_p$ is H{\"o}lder continuous
(Lemma \ref{lemma:reg_H}), $\nabla u$ is H{\"o}lder continuous (Lemma \ref{eq:HJB-unique-solution}), and $P$ is H{\"o}lder continuous by assumption. Thus $v$ is H{\"o}lder continuous.
Now we show that $D_x v \in \mathcal{C}^{\alpha}(Q,\R^{d\times d})$. The derivative of $v$ is given by
\begin{equation} \label{eq:derivative-v}
D_x v = - \bm{H}_{px}[\nabla u + A^\star P] - \bm{H}_{pp}[\nabla u + A^\star P](D^2_{xx} u + D_x A^\star P).
\end{equation}
Assumption \eqref{ass_holder} yields $D_x A^\star P \in \mathcal{C}^\alpha(Q;\mathbb{R}^{d \times d})$. In addition we have that $\nabla u \in \mathcal{C}^\alpha(Q;\mathbb{R}^d)$ and $ D_{xx}^2 u \in \mathcal{C}^\alpha(Q;\mathbb{R}^{d \times d})$. Finally the H{\"o}lder continuity of $H_{pp}$ (see Lemma \ref{lemma:reg_H}) yields  $v \in \mathcal{C}^{1+\alpha,\alpha}(Q)$. It follows that $m \in \mathcal{C}^{2+\alpha,1+ \alpha/2}(Q;\R^d)$, by Theorem \ref{theo:holder_reg_classical} and $w = m v \in \mathcal{C}^{1+\alpha,\alpha}(Q;\R^d)$, as was to be proved.

\noindent \textit{Point \ref{point:mvw-W-21p}.} The constants $C$ used for proving the second point only depend on $C$.
By Lemma \ref{lemma:reg_H}, $H_p$, $H_{pp}$, and $H_{px}$ are H{\"o}lder continuous.
By \eqref{ineq:assumption-lemma-gamma-P} and Lemma \ref{eq:HJB-unique-solution}, there exists $C>0$ only depending on $R$ such that $\|v\|_{L^\infty(Q;\mathbb{R}^d)} \leq C$.

We use again formula \eqref{eq:derivative-v} for proving that $D_x v$ is uniformly bounded. We know that $a$ and $D_x a$ are bounded (Assumption \eqref{ass_holder}) and by Lemma \ref{eq:HJB-unique-solution}, $\nabla u$ and $D_xx^2 u$ are bounded in $L^\infty$ by some constant depending on $R$.
We conclude that $\| v \|_{W^{1,0,\infty}(Q;\R^d)} \leq C$, for some $C$ depending only on $R$.
Now we have that $m$ is the solution to the Fokker-Planck equation
\begin{equation} \nonumber
\begin{array}{rlr}
\partial_t m - \Delta m + m (\nabla \cdot v) + \nabla m \cdot v = & 0, \quad & (x,t) \in Q, \\
m(0,x) = & m_0(x), & x \in \mathbb{T}^d.
\end{array}
\end{equation}
Since $\| v \|_{W^{1,0,\infty}(Q;\R^d)} \leq C$, we have that $m$ is the solution of a parabolic PDE with bounded coefficients, which implies that $\|m\|_{W^{2,1,p}(Q)} \leq C$,
by Theorem \ref{theo:max_reg1}.
By Lemma \ref{lemma:max_reg_embedding}, we have $ \| m \|_{L^\infty(Q)} \leq C $ and $\| \nabla m_k \|_{L^\infty(Q;\R^d)} \leq C$. It follows that  $\| w \|_{W^{1,0,\infty}(Q;\R^d)} \leq C$ since $w = mv$.

\noindent \textit{Point \ref{point-verif}.}
The statement holds by a classical verification argument. 

\noindent \textit{Point \ref{point-opti-mv-indiv}.}
This is a direct consequence of Point \ref{point-verif} and Lemma \ref{lemma:equality-of-cost}. Indeed, for any $(m',v') \in \mathcal{R}$, we have
\begin{equation} \nonumber
\mathcal{Z}_{\gamma,P}(m',v')
= Z_{\gamma,P}(\nu^{v'})
\geq Z_{\gamma,P}(\nu^v)
= \mathcal{Z}_{\gamma,P}(\mathsf{m}[v],v),
\end{equation}
which proves the optimality of $(\mathsf{m}[v],v)$. The optimality of $\chi (\mathsf{m}[v],v)$ follows then from Remark \ref{rem:equivalence}.
\end{proof}

\section{Properties of the solution to the mean field game system \label{section:MFG}}

We first recall the main result of \cite{BHP-schauder} concerning the existence and uniqueness of a solution $(\bar m,\bar v,\bar u,\bar \gamma,\bar P)$ to \eqref{eq:MFG-gcg}.
Then we establish a quadratic growth property (inequality \eqref{estim:quad-v}) which is at the heart of our convergence analysis in Section \ref{sec:convergence}.
It allows to show that $(\bar m, \bar v)$ is the unique solution to an optimization problem \eqref{pb:control-alpha-FP} and that $(\bar m, \bar m \bar v)$ is the unique solution to an equivalent convex potential problem \eqref{pb:control-w-FP}.
With an analogous reasonning, we prove the uniqueness of the solutions to  problems \eqref{pb:individual-control-problem} and \eqref{pb:individual-control-problem-w}.

\begin{theorem} \label{theo:main}
There exists $\alpha \in (0,1)$ such that \eqref{eq:MFG-gcg} has a unique classical solution $(\bar m,\bar v,\bar u,\bar \gamma,\bar P)$, with
\begin{equation} \label{eq:main}
\begin{cases}
\begin{array}{rl}
\bar m \in & \mathcal{C}^{2+\alpha,1+ \alpha/2}(Q), \\
\bar v \in & \mathcal{C}^{1+ \alpha,\alpha}(Q;\R^d)\\
\bar u \in & \mathcal{C}^{2+\alpha,1+ \alpha/2}(Q), \\
\bar \gamma \in & \mathcal{C}^{\alpha}(Q), \\
\bar P \in & \mathcal{C}^\alpha(0,T;\R^k).
\end{array}
\end{cases}
\end{equation}
\end{theorem}

\begin{proof}
Direct application of \cite[Theorem 1, Proposition 2]{BHP-schauder}.
\end{proof}

We define the following primal problem
\begin{equation} \label{pb:control-alpha-FP} \tag{P}
\inf_{(m,v) \in \mathcal{R}} \mathcal{J}(m,v) \coloneqq \int_{Q} \bm{L}[v] m \dd x \dd t + \int_{0}^{T}\left( \bm{F}[m] + \bm{\Phi}[A[m v]] \right)\dd t + \int_{\mathbb{T}^d} g m(T) \dd x.
\end{equation}

\begin{lemma} \label{lemma:J-J}
Let $(\bar{m},\bar{v}, \bar{u}, \bar{\gamma}, \bar{P})$ be the solution to \eqref{eq:MFG-gcg}. Then there exists a constant $C>0$ such that for any $(m,v) \in \mathcal{R}$ we have the following estimate:
\begin{equation}\label{estim:quad-v}
\mathcal{J}(m,v) - \mathcal{J}(\bar{m},\bar{v}) \geq \frac{1}{C} \int_Q |v - \bar{v}|^2 m \dd x \dd t.
\end{equation}
\end{lemma}

\begin{proof}
By \cite[Proposition 2]{BHP-schauder},  we have that $(\bar{m},\bar{v})$ is solution to Problem \eqref{pb:control-alpha-FP}.
By (\ref{eq:MFG-gcg},ii) we have that $\bar{v} = - \bm{H}_p[\nabla \bar{u} + A^\star \bar{P}]$. Then by Lemma \ref{lemma:H-L-quad},
\begin{align} \nonumber
\bm{L}[v](x,t) m(x,t) - \bm{L}[\bar{v}](x,t) \bar{m}(x,t) \geq  - \bm{H}[\nabla \bar{u} + A^\star \bar{P}](x,t) (m(x,t) - \bar{m}(x,t)) \\
- \langle (\nabla \bar{u} + A^\star \bar{P})(x,t),w(x,t) - \bar{w}(x,t) \rangle + \frac{1}{C}|v(x,t) - \bar{v}(x,t)|^2 m(x,t), \label{ineq-conv:L-L-application}
\end{align}
 for all $(x,t) \in Q$, $v \in \mathcal{C}^{1 + \alpha}(Q;\mathbb{R}^d)$ where $(w,\bar{w}) = (mv,\bar{m}\bar{v})$. By (\ref{eq:MFG-gcg},i),
\begin{equation}
\int_Q - \bm{H}[\nabla \bar{u} + A^\star \bar{P}] (m - \bar{m}) \dd x \dd t = \int_Q ( - \partial_t \bar{u} - \Delta \bar{u} - \bar{\gamma})(m - \bar{m})  \dd x \dd t. \label{ev:HJB}
\end{equation}
By (\ref{eq:MFG-gcg},iv) we have that $\bar{\gamma}(x,t) = f(x,t,\bar{m}(t))$ thus by convexity of $\bm{F}$,
\begin{equation} \label{ineq-conv:F-F}
\int_0^T \left(\bm{F}[m] -  \bm{F}[\bar{m}] \right) \dd t \geq \int_Q \bar{\gamma}(m - \bar{m}) \dd x \dd t.
\end{equation}
By (\ref{eq:MFG-gcg},v) we have that $\bar{P} = \bm{\phi}[A\bar{w}]$ thus by convexity of $\bm{\Phi}$,
\begin{equation}
\int_0^T \left( \bm{\Phi}[A[w]] -  \bm{\Phi}[A[\bar{w}]] \right) \dd t \geq  \int_0^T \langle \bar{P}, A[w - \bar{w}] \rangle  \dd t  =  \int_{Q} \langle  A^\star \bar{P}, w - \bar{w} \rangle \dd x \dd t. \label{ineq-conv:phi-phi}
\end{equation}
Combining \eqref{ineq-conv:L-L-application}, \eqref{ev:HJB},  \eqref{ineq-conv:F-F}, and \eqref{ineq-conv:phi-phi} and integrating by parts we obtain that
\begin{align} \nonumber
\mathcal{J}(m,v)  - \mathcal{J}(\bar{m},\bar{v})   \geq &  \int_Q  \left( (\partial_t \bar{u} - \Delta  \bar{u} ) (m - \bar{m}) -  \nabla \bar{u} (w - \bar{w})  \right) \dd x \dd t \\ \nonumber
& + \int_{\mathbb{T}^d} (m(T)-\bar{m}(T)) g \dd x + \frac{1}{C} \int_Q |v - \bar{v}|^2 m \dd x \dd t  \\ \nonumber
 \geq & \int_Q \bar{u} \left( \partial_t (m - \bar{m}) - \Delta (m - \bar{m}) +  \nabla \cdot (w - \bar{w}) \right) \dd x \dd t \\ \nonumber
& + \int_{\mathbb{T}^d} \bar{u}(0) (m(0) - m_0) \dd x+ \frac{1}{C} \int_Q |v - \bar{v}|^2 m \dd x \dd t .
\end{align}
Then \eqref{estim:quad-v} holds since $(m,w)$ and $(\bar{m},\bar{w})$ lie in $\tilde{\mathcal{R}}$.
\end{proof}

We next consider the problem
\begin{equation} \label{pb:control-w-FP} \tag{\~{P}}
\inf_{(m,w) \in \tilde{\mathcal{R}}} \tilde{\mathcal{J}}(m,w) \coloneqq  \int_{Q} \tilde{\bm{L}}[m,w] \dd x \dd t + \int_0^T \left(\bm{F}[m] + \bm{\Phi}[Aw] \right) \dd t + \int_{\mathbb{T}^d} g m(T) \dd x.
\end{equation}

\begin{corollary} \label{corollary:uniq}
Let $(\bar{m},\bar{v}, \bar{u}, \bar{\gamma}, \bar{P})$ be the unique solution to \eqref{eq:MFG-gcg}. Then $(\bar{m},\bar{v})$ is the unique solution to Problem \eqref{pb:control-alpha-FP} and $(\bar{m},\bar{w}) \coloneqq \chi(\bar{m},\bar{v})$ is the unique solution to Problem \eqref{pb:control-w-FP}.
\end{corollary}

\begin{proof}
Let $(m,v), (m',v') \in \mathcal{R}$ be two solutions to Problem \eqref{pb:control-alpha-FP}. Then by Lemma \ref{lemma:J-J} we have $\int_Q |v - v'|^2 m \dd x \dd t = 0$ which yields $v=v'$ since $m$ is positive. Then $m$ and $m'$ are solution to the same Fokker-Planck equation and thus $m = m'$.
Finally, $(\bar{m},\bar{w})$ is the unique solution to \eqref{pb:control-w-FP}, by Remark \ref{rem:equivalence}.
\end{proof}

\begin{lemma} \label{lemma:uniqueness-Z}
Let $\beta \in (0,1)$,
let $(\gamma,P) \in \mathcal{U}^\beta$, let $u = \mathsf{u}[\gamma,P]$ and let $v = - \bm{H}_p[\nabla u + P]$. Then $(\mathsf{m}[v],v)$ is the unique solution to Problem \eqref{pb:individual-control-problem} and $\chi(\mathsf{m}[v],v)$ is the unique solution to Problem \eqref{pb:individual-control-problem-w}.
\end{lemma}

\begin{proof}
The optimality of $(\mathsf{m}[v],v)$ and $\chi(\mathsf{m}[v],v)$ has been established in Lemma \ref{lemma:link-sto-edp}. Following the proof of Lemma \ref{lemma:J-J}, one can easily show that
\begin{equation} \nonumber
\mathcal{Z}_{\gamma,P}(m',v') - \mathcal{Z}_{\gamma,P}(\mathsf{m}[v],v) \geq \frac{1}{C} \int_Q |v' - v|^2 m' \dd x \dd t,
\end{equation}
for any $(m',v') \in \mathcal{R}$. Applying the same reasoning as in the proof of Corollary \ref{corollary:uniq} and using Remark \ref{rem:equivalence} allows to conclude the proof.
\end{proof}

\section{Generalized conditional gradient \label{sec:gcg}}

In this section we first present the generalized conditional gradient method in an abstract framework. Then we present a generalized conditional gradient method for our potential mean field game. We show that this procedure is linked with the fictitious play method, a learning procedure. The generalized conditional gradient point of view allows us to link two notions from different areas: the notion of exploitability from game theory and the notion of duality gap defined in (generalized) conditional gradient theory.

\paragraph{Abstract framework}

We present here the main ideas of the generalized conditional gradient method in a finite dimensional setting.
Consider the optimization problem
\begin{equation} \label{eq:pb_abstrait} \tag{P$_f$}
\min_{x \in K} f(x)= f_1(x) + f_2(x),
\end{equation}
where $K$ is a convex and compact subset of $\R^n$ of finite diameter $D$, $f_1$ is a (possibly non-smooth) convex function and $f_2$ a continuous differentiable function with $L$-Lipschitz gradient. 
We consider the mapping $h \colon K \times K \rightarrow \R$ defined by
\begin{equation} \nonumber
h(x,y) = f_1(y)-f_1(x) + \langle \nabla f_2(x),y-x \rangle.
\end{equation}
The mapping $h$ is a kind of first-order approximation of $f(y)-f(x)$, where only $f_2$ is linearized.
 Let $(\delta_{k})_{k \in \mathbb{N}} \in [0,1]$ be a sequence of step sizes.
The method generates iteratively two sequences $(\bar{x}_{k})_{k \in \mathbb{N}}$ and $(x_{k})_{k \in \mathbb{N}}$ in $K$. At iteration $k$, $\bar{x}_{k}$ is available and $(x_{k},\bar{x}_{k+1})$ is obtained as follows:
 \begin{align} \nonumber
& x_{k} \in \argmin_{y \in K} h(\bar{x}_{k},y), \\ \nonumber
& \bar{x}_{k+1} =  (1-\delta_{k}) \bar{x}_{k} + \delta_{k} x_{k}.
\end{align}
We also consider the mapping $\sigma \colon K \rightarrow \R$ defined by
\begin{equation} \nonumber
\sigma(x)= - \min_{y \in K} h(x,y) \geq 0.
\end{equation}
We call $\sigma(x)$ the primal-dual gap at $x\in K$. This terminology is motivated by the following. Consider the Lagrangian $\mathcal{L} \colon K \times \mathbb{R}^d \times \mathbb{R}^d \to \mathbb{R}$, 
\begin{equation} \nonumber
\mathcal{L}(x,y,\lambda) = f_1(x) + f_2(y) + \langle \lambda, x-y \rangle.
\end{equation}
It is easy to verify that \eqref{eq:pb_abstrait} can be formulated as follows:
\begin{equation*}
\inf_{x \in K,\, y \in \R^d} \ \sup_{\lambda \in \R^d}
\mathcal{L}(x,y,\lambda).
\end{equation*}
In particular, for $x \in K$, we have $f(x)= \sup_{\lambda' \in \R^d} \mathcal{L}(x,x,\lambda')$.
The dual problems writes
\begin{equation*}
\sup_{\lambda \in \R^d} \ \inf_{x \in K,\, y \in \R^d}
\mathcal{L}(x,y,\lambda).
\end{equation*}
Given $x \in K$, a candidate for the dual problem is $\lambda= \nabla f_2(x)$. The dual cost is then
\begin{align*}
\inf_{x' \in K,\, y' \in \R^d}
\mathcal{L}(x',y',\lambda)
= \ & \inf_{x' \in K} f_1(x') + \langle \lambda, x' \rangle
+ \inf_{y' \in \R^d} f_2(y') - \langle \nabla f_2(x), y' \rangle \\
= \ & f(x) + \inf_{x' \in K} h(x,x')
= f(x) - \sigma(x).
\end{align*}
Thus $\sigma(x)$ is nothing but the difference between the primal cost at $x$, and the dual cost at $\nabla f_2(x)$.
We will later see that it coincides with the notion of exploitability in the context of mean field games.

Under the previous assumptions, one can show that (see \cite[Lemma 2.4]{rakotomamonjy2015generalized})
\begin{equation} \label{eq:certificate-f}
0 \leq  f(\bar{x}_k) - f(\bar{x}) \leq  \sigma(\bar{x}_k),
\end{equation}
where $\bar{x}$ is a solution to problem \eqref{eq:pb_abstrait}. In words, any point $x \in K$ is $\sigma(x)$-optimal.

\paragraph{Application to potential mean field games}

Our framework is infinite dimensional, we aim at minimizing the potential $\tilde{\mathcal{J}}(m,w)$ under the constraint $(m,w) \in \tilde{\mathcal{R}}$. Following the ideas presented in the previous paragraph, we define a mapping $h : \tilde{\mathcal{R}} \times \tilde{\mathcal{R}} \to \mathbb{R}$,
\begin{align} \notag
 h((m,w),(m',w')) = \ & \tilde{\mathcal{Z}}_{\gamma,P}(m',w') - \tilde{\mathcal{Z}}_{\gamma,P}(m,w)\\  =\ &\int_{Q} \left( \tilde{\bm{L}}[m',w']- \tilde{\bm{L}}[m,w]  \right) \dd x \dd t + \int_{\mathbb{T}^d}g (m'-m)(T) \dd x  \nonumber \\ & + \int_{Q} \gamma (m'-m) \dd x \dd t + \int_0^T \langle A[w'-w], P \rangle \dd t \label{def:h}
\end{align}
where $\gamma(x,t) = f(x,t,m(t))$ and $P(t) = \phi(t,A w (t))$ for any $(x,t) \in Q$. 
By analogy with the previous abstract framework, we can interpret $h((m,w),(m',w'))$ as a partial linearization of $\tilde{\mathcal{J}}(m',w') - \tilde{\mathcal{J}}(m,w)$:
we have a non-linearized part composed of the perspective function $\tilde{\bm{L}}$ (analogous to the term $f_1$) and a linearized part composed of all the other terms (analogous to the term $f_2$): the congestion $\gamma$, the price $P$ and the terminal cost $g$. Two reasons motivates this choice of linearization: 
\begin{enumerate}
\item In general the perspective function $\tilde{\bm{L}}$ is not differentiable.
\item This particular choice of linearization allows to link the generalized conditional gradient method with the fictitious play algorithm, 
as explained in the end of this section.
\end{enumerate}

We define the following generalized conditional gradient algorithm for potential mean field games as follows:

\begin{algorithm}[H]
\caption{Generalized conditional gradient} \label{algo:gcg}
\begin{algorithmic}
\STATE Choose $(\bar{m}_0,\bar{w}_0) \in \mathcal{C}^{2+\alpha,1+ \alpha/2}(Q) \times   \mathcal{C}^{1+ \alpha,\alpha}(Q;\R^d)$ with $\bar{m}_0(x,t) >0$ for any $(x,t) \in Q$ and choose a sequence $(\delta_k)_{k\in \mathbb{N}} \in [0,1]$.
\FOR{$0\leq k < N$} 
\STATE {Find $(m_k,w_k) = \argmin_{(m,w) \in \tilde{\mathcal{R}}} h((\bar{m}_k,\bar{w}_k),(m,w))$}
\STATE {Actualise  $(\bar{m}_{k+1},\bar{w}_{k+1}) = (1- \delta_{k}) (\bar{m}_{k},\bar{w}_{k}) + \delta_{k} (m_{k},w_{k})$}
\ENDFOR
\RETURN {$(\bar{m}_N, \bar{w}_N)$.}
\end{algorithmic}
\end{algorithm}

\begin{figure}[htb]
\centering
\includegraphics[scale=1]{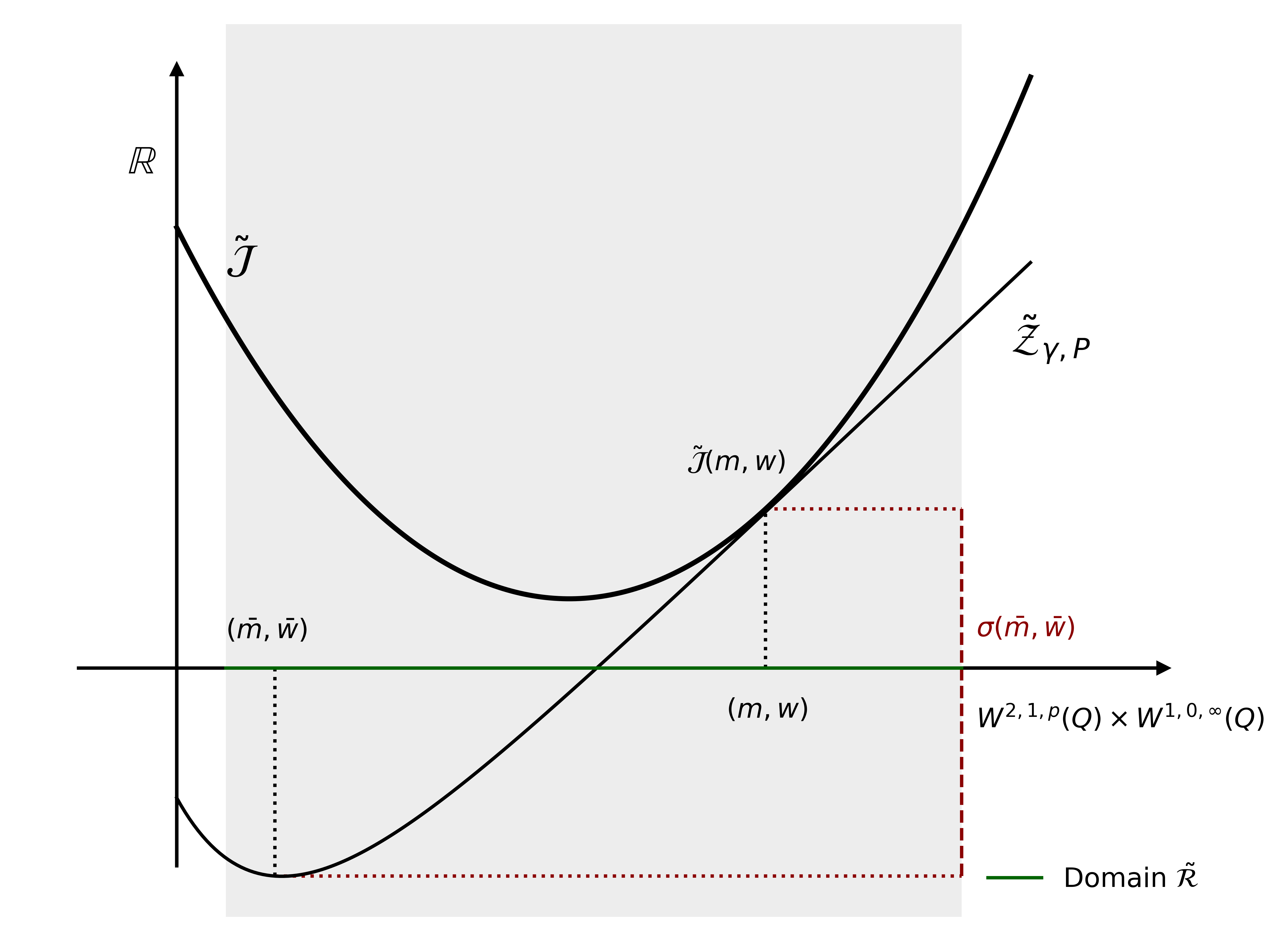} 
\caption{Illustration of the potential cost $\tilde{\mathcal{J}}$, the individual cost $\tilde{\mathcal{Z}}_{\gamma,P}$ and the exploitability $\sigma$.}
\end{figure}

We first justify the well-posedness of the algorithm (in particular, we need to justify the existence and uniqueness of $(m_k,w_k)$).
To this goal, we introduce the following sequences
\begin{equation} \nonumber
\begin{array}{llll}
 P_k(t) & = \phi(t,A \bar{w}_k (t)), & \gamma_k(x,t) & = f(x,t,\bar{m}_k(t)), \\
 u_k(x,t) & = \mathsf{u}[ \gamma_k, P_k](x,t), \qquad & v_k(x,t) & = - \bm{H}_p[ \nabla u_k + A^\star P_k]](x,t),
\end{array}
\end{equation}
for any $(x,t) \in Q$.
For future reference, we define
\begin{equation*}
\bar{v}_k = \bar{w}_k / \bar{m}_k.
\end{equation*}
In the next lemma, we provide an explicit formula to the minimization step, directly derived from Lemma \ref{lemma:link-sto-edp}.

\begin{lemma} \label{lemma:fictitious-play-seq}
For all $k\in \mathbb{N}$, we have $(m_k,w_k) = \chi(\mathsf{m}[v_k],v_k)$. Moreover, there exists $\alpha_k \in (0,1)$ such that
\begin{equation} \label{reg:fictitious-play-seq}
\begin{cases}
\begin{array}{rl}
m_k, \bar{m}_k \in & \mathcal{C}^{2+\alpha_k,1+ \alpha_k /2}(Q), \\
v_k, w_k, \bar{w}_k \in & \mathcal{C}^{1+ \alpha_k,\alpha_k}(Q;\R^d),\\
u_k \in & \mathcal{C}^{2+\alpha_k,1+ \alpha_k/2}(Q), \\
\gamma_k \in & \mathcal{C}^{\alpha_k}(Q), \\
P_k \in & \mathcal{C}^{\alpha_k}(0,T;\R^k).
\end{array}
\end{cases}
\end{equation}
\end{lemma}

\begin{proof}
We prove the result by induction.
Let $k \in \mathbb{N}$. Assume that there exists $\alpha \in (0,1)$ such that $\bar{m}_k \in \mathcal{C}^{2+\alpha,1+ \alpha/2}(Q)$, $\bar{w}_k \in \mathcal{C}^{1+ \alpha,\alpha}(Q;\R^d)$.

\noindent \textit{Step 1: $P_k \in  \mathcal{C}^{\alpha}(0,T;\R^k)$ and $\gamma_k \in \mathcal{C}^{\alpha}(Q)$.} By assumptions \eqref{ass_holder} and \eqref{ass_Lip_phi},
\begin{align} \nonumber
|P_k(t_2) - P_k(t_1) |  & =  |\phi[A\bar{w}_k](t_2) -  \phi[A\bar{w}_k](t_1)|  \\ \nonumber & \leq C\left( |t_2 - t_1|^\alpha + |A\bar{w}_k(t_2)  - A \bar{w}_k(t_1)| \right)\\ \nonumber
& \leq  C \left(|t_2 - t_1|^\alpha + \| a \|_{L^\infty(Q;\mathbb{R}^{k\times d})} \int_{\mathbb{T}^d} | \bar{w}_k(t_2)  - \bar{w}_k(t_1)|\dd x\right),
\end{align}
for all $t_1,t_2 \in [0,T]$.
It follows that $P_k$ is H\"older continuous, since by induction assumption, $\bar{w}_k \in \mathcal{C}^{1+ \alpha,\alpha}(Q;\R^d)$.
The announced regularity on $\gamma_k$ is a direct consequence of the induction assumption ($\bar{m}_k \in \mathcal{C}^{2+\alpha,1+ \alpha/2}(Q)$) and Assumption \eqref{ass_hold_f}.

\noindent \textit{Step 2: $u_k \in \mathcal{C}^{2+\alpha,1 + \alpha/2}(Q)$.} The regularity of $P_k$ and $\gamma_k$ obtained in the previous steps allows us to apply \ref{eq:HJB-unique-solution}, which yields the announced regularity on $u_k$.

\noindent \textit{Step 3: $(m_k,w_k) = \chi(\mathsf{m}[v_k],v_k)$}.
By Lemma \ref{lemma:link-sto-edp} and Lemma \ref{lemma:uniqueness-Z}, $\chi(\mathsf{m}[v_k],v_k)$ is the unique minimizer of $\mathcal{Z}_{\gamma_k,P_k}$, thus the unique minimizer of $h((\bar m_k, \bar w_k), \cdot )$ on $\tilde{\mathcal{R}}$.

\noindent \textit{Step 4: $v_k \in \mathcal{C}^{1+ \alpha,\alpha}(Q;\R^d)$,  $m_k \in \mathcal{C}^{2+\alpha,1+ \alpha/2}(Q)$, and $w_k \in \mathcal{C}^{1+\alpha,\alpha}(Q;\R^d)$.} 
Direct consequence of the previous steps and Point \ref{point:vmw-c-alpha} of Lemma \ref{lemma:link-sto-edp}.

\noindent \textit{Conclusion.}
By Step 4 and by the induction assumption, we have that $(\bar{m}_{k+1},\bar{w}_{k+1}) \in \mathcal{C}^{2+\alpha,1+ \alpha/2}(Q) \times \mathcal{C}^{1+\alpha,\alpha}(Q;\R^d)$. Thus the induction assumption holds at $k+1$, which concludes the proof.
\end{proof}

\paragraph{Link with the fictitious play}

Let us consider the primal-dual gap
\begin{equation} \label{eq:def_exploitability}
\sigma_k
= - \min_{(m,w) \in \tilde{\mathcal{R}}}  h((\bar m_k, \bar w_k),(m,w)).
\end{equation}
As mentioned earlier, $\sigma_k$ is a primal gap certificate; it provides us with an upper bound of $\tilde{\mathcal{J}}(\bar m_k, \bar w_k)-\tilde{\mathcal{J}}(\bar m, \bar w)$ (this will be proved in Lemma \ref{eq:seq_gamma}). In the current mean field game context, it coincides with the notion of exploitability: it is the largest decrease in cost that a representative agent can reach by playing its best response, assuming that all other agents use the feedback $\bar{v}_k:= \bar{w}_k / \bar{m}_k$.
Indeed, we have
\begin{align*} \nonumber
\sigma_k = \ &
\tilde{\mathcal{Z}}_{\gamma_k,P_k}(\bar m_k, \bar w_k) - \inf_{(m,w) \in \tilde{\mathcal{R}}} \tilde{\mathcal{Z}}_{\gamma_k,P_k}(m,w) \\
= \ &
Z_{\gamma_k,P_k}(\nu^{\bar{v}_k}) - \inf_{\nu \in L_{\mathbb{F}}^2(t,T;\mathbb{R}^d)} Z_{\gamma_k,P_k}(\nu),
\end{align*}
by Lemma \ref{lemma:link-sto-edp} and Lemma \ref{lemma:fictitious-play-seq}.


We provide now an interpretation of the generalized gradient algorithm as a learning procedure called fictitious play. A definition and a study of the latter learning algorithm in the context of mean field games can be found in \cite{cardaliaguet2017learning,HADIKHANLOO2019369}.
Each iteration $k$ of Algorithm \ref{algo:gcg} relies on the following steps:

For $k \in \mathbb{N}$ let $(\bar{m}_k,\bar{w}_k)$ be a given belief and $\gamma_k$ and $P_k$ the resulting beliefs on congestion and price. Then there are four main steps:

\begin{enumerate}
\item Given $(\bar m_k, \bar w_k)$ compute the congestion terms $P_k$ and $\gamma_k$. In words, the agents make a prediction of the congestion term and the price at equilibrium, based on the belief $(\bar m_k, \bar w_k)$. 

\item Find the value function $u_k$ solution to the Hamilton-Jacobi-Bellman equation parametrized by $(\gamma_k,P_k)$. Then compute the optimal control $v_k$, given the value function $u_k$ and the price $P_k$. 
This step can be interpreted as follows: for a given belief on the distributions of the others $(m_k,w_k)$, a representative agent computes its best response $v_k$.
\item Find the solution $m_k$ to the Fokker-Planck equation for the given drift $v_k$ and compute the associated distribution of controls $w_k$.
\item The actualization step of  $(\bar{m}_{k+1},\bar{w}_{k+1})$ can be interpreted as a learning step. The learning rule consists in averaging the past realizations of the distribution and flow at a rate determined by the sequence $(\delta_k)_{k \in \mathbb{N}}$. 
\end{enumerate}

\section{Convergence Results \label{sec:convergence}}

In this section, the generic constants $C$ and $\alpha$ depend on the data of the problem (introduced in Section \ref{subsec:coupled_system}) and depend on the pair $(\bar{m}_0,\bar{w}_0)$ chosen to initialize Algorithm \ref{algo:gcg}.

\begin{lemma} \label{bound:seq}
There exists $C > 0$ such that for any $k \in \mathbb{N}$,
\begin{equation} \nonumber
\begin{array}{llll}
\| \gamma_k \|_{W^{1,0,\infty}(Q;\R^d)} & \leq C \qquad \qquad
& \| m_k \|_{W^{2,1,p}(Q)} & \leq C \\
\| P_k \|_{L^{\infty}(0,T;\R^k)} & \leq C
& \| w_k \|_{W^{1,0,\infty}(Q;\R^d)} & \leq C \\
\| u_k \|_{W^{2,1,p}(Q)} & \leq C
& \| \bar{m}_k \|_{W^{2,1,p}(Q)} & \leq C \\
\| \nabla u_k \|_{W^{2,1,p}(Q;\R^d)} & \leq C
& \| \bar{w}_k \|_{W^{1,0,\infty}(Q;\R^d)} & \leq C. \\
\| v_k \|_{W^{1,0,\infty}(Q;\R^d)} & \leq C \\
\end{array}
\end{equation}
In addition, we have
\begin{equation} \nonumber
m_k(x,t) \geq 1/C, \quad \bar{m}_k(x,t) \geq 1/C, \quad \| \bar{v}_k \|_{L^\infty(Q;\mathbb{R}^d)} \leq C, \quad 
\end{equation}
for all $(x,t) \in Q$.
\end{lemma}

\begin{proof}
Let $k \in \mathbb{N}$.
Assume that there exists $C>0$ such that the bounds hold for all $i \in \left\{0,\ldots,k-1\right\}$.

\noindent \textit{Step 1: Bounds of $\gamma_k$ and $P_k$.}  These bounds directly follow from assumptions \eqref{ass_psi_bounded}, \eqref{ass_f_bounded}, and \eqref{ass_hold_f}.
They imply the existence of $C > 0$ such that
\begin{equation*}
\|\gamma_k\|_{L^\infty(Q)} + \|\nabla \gamma_k\|_{L^\infty(Q;\mathbb{R}^d)}  + \|P_k\|_{L^\infty(0,T;\mathbb{R}^k)} \leq C,
\end{equation*}
so that we can employ the technical Lemmas of Section \ref{sec:stochastic-pb} to prove the other announced bounds.

\noindent \textit{Step 2: Bounds of $u_k$ and $\nabla u_k$.}  Direct consequence of Step 1 and Lemma \ref{eq:HJB-unique-solution}.

\noindent \textit{Step 3: Bounds of $v_k$, $m_k$ and $w_k$.}  Direct consequence of the previous steps and Point \ref{point:mvw-W-21p} of Lemma \ref{lemma:link-sto-edp}.

\noindent \textit{Step 4:} Bounds of $\bar{m}_k$ and $\bar{w}_k$.
This is a direct consequence of the fact that $(\bar{m}_k,\bar{w}_k)$ can be expressed as a convex combination of $(m_k, w_k)_{i=0,...,k-1}$ and $(\bar{m}_0,\bar{w}_0)$.

\noindent \textit{Step 5: $m_k(x,t), \bar{m}_k(x,t) \geq 1/C$ for any $(x,t) \in Q$.}
Since $m_k = \mathsf{m}[v_k]$ with $\| v_k \|_{W^{1,0,\infty}(Q;\R^d)} \leq C$ and $m_0(x) \geq \varepsilon_0$ for any $x \in \mathbb{T}^d$, therefore  $m_k(x,t)\geq 1/C$ by Lemma \ref{lemma:maximum-principle}. Then $\bar{m}_k(x,t) \geq 1/C$ as a convex combination of $(m_k)_{i=0,...,k-1}$.

\noindent \textit{Step 6: $\| \bar{v}_k \|_{L^\infty(Q;\mathbb{R}^d)} \leq C$.} By Step 4 and Step 5,
\begin{equation} \nonumber
\| \bar{v}_k \|_{L^\infty(Q;\mathbb{R}^d)}  = \left\| \bar{w}_k / \bar{m}_k \right\|_{L^\infty(Q;\mathbb{R}^d)} \leq C.
\end{equation} 

\noindent \textit{Conclusion.}
 Since  $(\bar{m}_0,\bar{w}_0) \in \mathcal{C}^{2+\alpha,1+ \alpha/2}(Q) \times   \mathcal{C}^{1+ \alpha,\alpha}(Q;\R^d)$ with $\bar{m}_0(x,t) >0$ for any $(x,t) \in Q$, the conclusion follows by induction.
\end{proof}

Recall the definition of the exploitability $\sigma_k$, given in \eqref{eq:def_exploitability}.
We define the sequence of primal gaps $(\epsilon_{k})_{k \in \mathbb{N}}$ as follows
\begin{equation} \nonumber \epsilon_{k} = \tilde{\mathcal{J}}(\bar{m}_k,\bar{w}_k) - \tilde{\mathcal{J}}(\bar{m},\bar{w}).
\end{equation}
We recall that $(\bar{m},\bar{w}) = \argmin_{(m,w) \in \tilde{\mathcal{R}}} \tilde{\mathcal{J}}(m,w)$. The following Lemma is a certificate result, similar to inequality \eqref{eq:certificate-f}.

\begin{lemma} \label{eq:seq_gamma}
We have that $\epsilon_{k} \leq \sigma_{k}$.
\end{lemma}

\begin{proof}
For any $(m,w) \in \tilde{\mathcal{R}}$ we have that
\begin{align} \nonumber
h((\bar{m}_{k},\bar{w}_{k}), (m,w))  = \tilde{\mathcal{J}}(m,w) -  \tilde{\mathcal{J}}(\bar{m}_{k},\bar{w}_{k}) + a + b,
\end{align}
where
\begin{align}
a & =  \int_0^T \bm{F}[\bar{m}_{k}] -  \bm{F}[m]  \dd t  +  \int_{Q} f(x,t,\bar{m}_k(t)) (m(x,t) - \bar{m}_k(x,t)) \dd x \dd t \leq 0,  \nonumber\\
b & = \int_0^T \bm{\Phi}[A \bar{w}_{k}]  -  \bm{\Phi}[Aw] \dd t + \int_0^T \langle \phi(t,A w_k(t)) , A[ w - \bar{w}_k](t) \rangle  \dd t \leq 0, \nonumber
\end{align}
by convexity of $\bm{F}$ and $\bm{\Phi}$. Then we have that
\begin{equation}
\inf_{(m,w) \in \tilde{\mathcal{R}}} h((\bar{m}_{k},\bar{w}_{k}), (m,w))   \leq \inf_{(m,w) \in \tilde{\mathcal{R}}}  \tilde{\mathcal{J}}(m,w) - \tilde{\mathcal{J}}(\bar{m}_{k},\bar{w}_{k}),
\end{equation}
and the conclusion follows.
\end{proof}

\begin{lemma} \label{lemma:upper_bound}
There exists $C>0$ such that for any $\delta \in [0,1]$, it holds:
\begin{equation} \label{eq:upper_bound}
\tilde{\mathcal{J}}(\bar{m}_k^\delta,\bar{w}_k^\delta) \leq \tilde{\mathcal{J}}(\bar{m}_{k},\bar{w}_{k}) - \delta \sigma_k + \delta^2 C,
\end{equation}
where $(\bar{m}_k^{\delta},\bar{w}_k^{\delta}) = \delta (m_k,w_k) + (1-\delta) (\bar{m}_k,\bar{w}_k)$.
\end{lemma}

\begin{proof}
The convexity of $\tilde{L}$ yields 
\begin{equation}
\int_{Q} \tilde{\bm{L}}[\bar{m}_k^\delta,\bar{w}_k^\delta] \dd x \dd t  \leq \int_{Q} \tilde{\bm{L}}[\bar{m}_{k},\bar{w}_{k}] + \delta \left(\tilde{\bm{L}}[\bar{m}_{k},\bar{w}_{k}] - \tilde{\bm{L}}[m_{k},w_{k}]  \right) \dd x \dd t. \label{ineq:L-L-conv}
\end{equation}
Using that $F$ is the primitive of $f$ in the sense of \eqref{primitive:F}, we have for all $t \in [0,T]$,
\begin{align} \nonumber
& \bm{F}[\bar{m}^{\delta}_{k}](t) = \bm{F}[\bar{m}_{k}](t)  \\
& + \delta_k \int_0^1 \int_{\mathbb{T}^d} f \left(x,t, \bar{m}_k(t) + s \delta (m_k(t)-\bar{m}_{k}(t))\right)(m_k(x,t) - \bar{m}_{k}(x,t)) \dd x \dd s, \label{ineq:F-F'}
\end{align}
For any $(x,t) \in Q$, the Lipschitz-continuity of $f$ yields
\begin{align} \nonumber
f\left(x,t,  \bar{m}_k(t) + s \delta (m_k(t)-\bar{m}_{k}(t)) \right) & \leq f(x,t,\bar{m}_k(t)) +  s \delta C \|m_k(t)-\bar{m}_k(t)\|_{L^2(\mathbb{T}^d)} \\
& \leq f(x,t,\bar{m}_k(t)) +  s \delta C, \nonumber
\end{align}
since $\bar{m}_k,m_k$ are uniformly bounded by Lemma \ref{bound:seq}. Plugging into \eqref{ineq:F-F'} yields
\begin{equation} \label{ineq:F-F-K}
\bm{F}[\bar{m}^{\delta}_{k}](t) = \bm{F}[\bar{m}_{k}](t)   
+  \delta \int_{\mathbb{T}^d} f(x,t,\bar{m}_k(t)) (m_k(x,t) - \bar{m}_{k}(x,t)) \dd x + \delta^2 C.
\end{equation}
Now using that $\Phi$ is the primitive of $\phi$, we have
\begin{equation} \nonumber
\bm{\Phi}[A \bar{w}^{\delta}_{k}](t) \leq \bm{\Phi}[A \bar{w}_k](t) + \delta \langle \phi(t,A  \bar{w}_k(t)) , A[ w_k - \bar{w}_k](t) \rangle 
+ \delta^2|A[w_k - \bar{w}_k](t)|^2
\end{equation}
by Assumption \eqref{ass_Lip_phi}. Using that $\bar{w}_k,w_k$ are uniformly bounded by Lemma \ref{bound:seq} yields 
\begin{equation} \nonumber
|A[ w_k - \bar{w}_k](t)| \leq \|a(t)\|_{L^\infty(\mathbb{T}^d)} \|w_k(t) - \bar{w}_k(t)\|_{L^\infty(\mathbb{T}^d)} \leq C.
\end{equation}
Combining the two last inequalities yields
\begin{align} \label{ineq:phi-phi-K}
\bm{\Phi}[A \bar{w}^{\delta}_{k}](t)  \leq \bm{\Phi}[A \bar{w}_k](t) + \delta \langle \phi(t,A  \bar{w}_k(t)) , A[ w_k - \bar{w}_k](t) \rangle + \delta^2 C.
\end{align}
Then inequality \eqref{eq:upper_bound} holds combining the Assumption \eqref{ass_init_cond} on $g$ and inequalities \eqref{ineq:L-L-conv}, \eqref{ineq:F-F-K}, and \eqref{ineq:phi-phi-K} which concludes the proof.
\end{proof}

\begin{lemma} \label{lemma:seq_varepsilon}
We have that 
\begin{equation} \nonumber
\epsilon_{k+1}  \leq (1- \delta_{k}) \epsilon_{k} +  \delta_{k}^2 C.
\end{equation}
\end{lemma}
\begin{proof}
A direct application of Lemma \ref{lemma:upper_bound} yields,
\begin{equation} \nonumber
\tilde{\mathcal{J}}(\bar{m}_{k+1},\bar{w}_{k+1}) \leq \tilde{\mathcal{J}}(\bar{m}_{k},\bar{w}_{k}) - \delta_k \sigma_k + \delta_k^2 C.
\end{equation}
Thus $\epsilon_{k+1}  \leq \epsilon_{k} - \delta_{k} \sigma_k + \delta_{k}^2 K$
and the conclusion follows by Lemma \ref{eq:seq_gamma} since $- \sigma_k \leq -\epsilon_{k}$.
\end{proof}

\begin{lemma} \label{lemma:rate-conv}
Let  $L_{0} \coloneqq \max\{\epsilon_0/2, C\}$ and $L_{1} \coloneqq \max\{2 \epsilon_1, C\}/\ln(2)$. We have that
\begin{equation} \label{eq:xi}
\begin{cases}
\begin{array}{llll}
\normalfont{\text{(i)}} & \epsilon_{k} \leq \frac{4 L_{0}  }{k+2}  &\normalfont{\text{if }} \delta_{k} = \frac{2}{k+2}, & \normalfont{\text{for any }} k \in \mathbb{N},  \\[1em]
\normalfont{\text{(ii)}}  &  \epsilon_{k} \leq \frac{\ln(k+1) L_{1}}{k+1} &\normalfont{\text{if }} \delta_{k} = \frac{1}{k+1}, & \normalfont{\text{for any }} k \in \mathbb{N}\setminus\{0\}.
\end{array}
\end{cases}
\end{equation}
\end{lemma}
The above Lemma summarizes the rate of convergence of the sequence $(\epsilon_{k})_{k\in \mathbb{N}}$ for two learning rates. The first result (\ref{eq:xi},i) is classical in the context of conditional gradient algorithm (see \cite{dunn1978conditional,frank1956algorithm}). For the sake of completeness we recall how to derive this result in the following proof. The second result (\ref{eq:xi},ii) corresponds to the classical fictitious play learning rate.
\begin{proof}
\textit{Step 1: (\ref{eq:xi},i) holds.}  Let $ \delta_{k} =2/(k+2)$ for any $k \in \mathbb{N}$. For $k = 0$, it is clear that (\ref{eq:xi},i) holds. For $k>0$, assume that $\epsilon_{k}$ satisfies the inequality  (\ref{eq:xi},i). By Lemma \ref{lemma:seq_varepsilon} we have that
\begin{equation} \nonumber
\epsilon_{k+1}  \leq \left(1- \frac{2}{k+2}\right)\frac{4 L_{0} }{k+2}+  \frac{4 C}{(k+2)^2} 
 \leq  \frac{4 L_{0} (k + 1)}{(k+2)^2} 
 \leq  \frac{4 L_{0}}{(k+3)},
\end{equation}
and by induction the step $1$ is proved.

\noindent \textit{Step 2: (\ref{eq:xi},ii) holds.} Let $ \delta_{k} = 1/(k+1)$ for any $k \in \mathbb{N}$. For $k = 1$, it is clear that (\ref{eq:xi},ii) holds by Lemma \ref{lemma:seq_varepsilon}. For $k>1$ assume that $\epsilon_{k}$ satisfies the inequality  (\ref{eq:xi},ii) then by Lemma \ref{lemma:seq_varepsilon} we have
\begin{equation} \nonumber
\epsilon_{k+1}  \leq \left(1- \frac{1}{k+1}\right)\frac{\ln(k+1) L_{1} }{k+1} +  \frac{C}{(k+1)^2}.
\end{equation}
Then to prove (\ref{eq:xi},ii) it is enough to check
\begin{equation} \nonumber
\frac{k \ln(k+1)L_1 + C}{(k+1)^2} \leq  \frac{\ln(k+2)}{(k+2)}.
\end{equation}
Multiplying both side by $(k+1)^2(k+2)$, the inequation (\ref{eq:xi},ii) holds if
\begin{align}
 (k+2)\frac{C}{L_1} & \leq  (k+1)^2 \ln(k+2) -  k(k+2)\ln(k+1) \nonumber\\ 
 & \leq k(k+2)\ln\left(1 + \frac{1}{k+1}\right)  + \ln(k+2). \label{ineq:ineq-ln1k1}
\end{align}
The concavity of the logarithm yields $ \ln\left(1 + 1/(k+1)\right)  \geq \ln(2)/k$. 
Thus the inequality \eqref{ineq:ineq-ln1k1} holds whenever
\begin{equation}  \nonumber
\frac{C}{L_1} \leq \ln(2) + \frac{\ln(k+2)}{k+2},
\end{equation}
which holds by definition of $L_1$. Then  Step 2 is proved, which concludes the proof.
\end{proof}

\begin{lemma} \label{lemma:exploitability}
There exists $C>0$ such that $\sigma_k \leq C \epsilon_k^{1/2}$ for all $k \in \mathbb{N}$.
\end{lemma}

\begin{proof}
For any $\delta \in [0,1]$, Lemma \ref{lemma:upper_bound} yields
$\tilde{\mathcal{J}}(\bar{m},\bar{w}) \leq \tilde{\mathcal{J}}(\bar{m}_{k},\bar{w}_{k}) - \delta \sigma_k + \delta^2 C$.
It follows that
\begin{equation} \label{eq:exploitability_cv}
\sigma_k \leq \epsilon_k/\delta  + \delta C, \qquad \forall \delta \in (0,1],
\end{equation}
by  Lemma \ref{lemma:rate-conv}. The optimal choice of $\delta \in (0,1]$ in the latter inequality is given by $\delta = \min \{\sqrt{\epsilon_k/C},1\}$. Since the sequence $(\epsilon_k)_{k\in \mathbb{N}}$ is uniformly bounded from above,
we can increase the constant $C$, so that one can choose $\delta= \sqrt{\epsilon_k/C} \in (0,1]$. For this choice of $\delta$, inequality \eqref{eq:exploitability_cv} yields the announced result.
\end{proof}

For any $k \in \mathbb{N}$ we denote 
\begin{equation} \nonumber
\begin{array}{llllll}
 \delta \bar{m}_k & = \bar{m}_k - \bar{m}, \qquad & \delta \bar{w}_k &= \bar{w}_k - \bar{w}, \qquad & \delta \bar{v}_k & = \bar{v}_k - \bar{v}, \\
 \delta P_k &= P_k - \bar{P}, &
\delta \gamma_k &= \gamma_k - \bar{\gamma} ,& \delta u_k &= u_k - \bar{u}.
\end{array}
\end{equation}

\begin{theorem} \label{lemma:conv-seq}
There exists $C>0$ such that for all $k \in \mathbb{N}$,
\begin{align} \nonumber
\| \delta \bar{v}_k \|_{L^2(Q;\mathbb{R}^d)} + \| \delta \bar{m}_k \|_{L^\infty(0,T;L^{2}(\mathbb{T}^d))} + \| \delta \bar{w}_k\|_{L^{2}(Q;\mathbb{R}^d)}  &  \\ +  \|\delta P_k \|_{L^2(0,T;\mathbb{R}^k)}  + \|\delta \gamma_k \|_{L^\infty(Q)} + \|\delta u_k  \|_{L^\infty(Q)} & \leq  C \epsilon^{1/2}_k.  \nonumber
\end{align}
\end{theorem}

\begin{proof}
\noindent \textit{Step 1: $\| \delta \bar{v}_k \bar{m}_k \|_{L^2(Q;\mathbb{R}^d)} \leq C \epsilon_k^{1/2}$.} 
By Lemma \ref{lemma:J-J}, we have
\begin{equation} \nonumber
\frac{1}{C} \int_Q |\delta \bar{v}_k|^2 \bar{m}_k \dd x \dd t \leq
\mathcal{J}(\bar{m}_k,\bar{v}_k) - \mathcal{J}(\bar{m},\bar{v}) = \epsilon_k.
\end{equation}
Combining the above inequality with $ \|\bar{m}_k\|_{L^\infty(Q)} \leq C$ yields the desired estimate.

\noindent \textit{Step 2: $\| \delta \bar{v}_k \|_{L^2(Q;\mathbb{R}^d)} \leq C \epsilon_k^{1/2}$.} 
By Step 1 and Lemma \ref{bound:seq},
\begin{equation} \nonumber
\| \delta \bar{v}_k \|_{L^2(Q;\mathbb{R}^d)} \leq \| \delta \bar{v}_k \bar{m}_k \|_{L^2(Q;\mathbb{R}^d)} \| 1/\bar{m}_k  \|_{L^\infty(Q)} \leq C \epsilon_k^{1/2},
\end{equation}
and Step 2 holds.

\noindent \textit{Step 3: $\|  \delta \bar{m}_k\|_{C(0,T;L^{2}(\mathbb{T}^d))}  \leq  C \epsilon^{1/2}_k$.}
We have that $ \delta \bar{m}_k$ satisfies
\begin{equation} \nonumber
\begin{array}{rlr}
\partial_t \delta \bar{m}_k  - \Delta \delta \bar{m}_k  + \nabla \cdot ( \bar{v} \delta \bar{m}_k )  = & -  \nabla \cdot (\delta \bar{v}_k  \bar{m}_k), \quad & (x,t) \in Q, \\
\delta m_k(0,x)  = & 0, & x \in \mathbb{T}^d.
\end{array}
\end{equation}
We define the space $V = W^{2,1}(\mathbb{T}^d)$ and its dual $V^{*}$. Then $\delta m_k$ is solution of a parabolic equation of the form
\begin{equation} \nonumber
\begin{array}{rlr}
 \partial_t  m(t) + B(t) m(t) = & f(t), \quad & (x,t) \in Q, \\
m(0,x) = & 0, & x \in \mathbb{T}^d,
\end{array}
\end{equation}
where $B(t) \in L(V, V^{*})$ and $f(t) \in V^{*}$. It is easy to verify that since $\bar{v} \in W^{1,0,\infty}(Q;\R^d)$, there exists a constant $C$ such that $\langle B(t)y,y' \rangle_V \leq C \| y \|_V \| y' \|_V$, for a.e.\@ $t \in (0,T)$ and for all $y$ and $y'$ in $V$.
For any $y \in V$ we further have that
\begin{align} \nonumber
\langle B(t) y, y \rangle_{V} & = \int_{\mathbb{T}^d} \left(- \Delta y + \nabla \cdot \bar{v}(t) y + \langle \bar{v}(t), \nabla y \rangle \right) y \dd x \\ \nonumber
& \geq \int_{\mathbb{T}^d} |\nabla y|^2  + C |y|^2 -   C |\nabla y |  |y| \dd x  \geq \frac{1}{2} \|y \|^2_{V} - \frac{C}{2}  \|y \|^2_{L^2(\mathbb{T}^d)},
\end{align}
where we have used that $ - \int_{\mathbb{T}^d} | \nabla y  | |y| \dd x  \geq - \frac{1}{2} \int_{\mathbb{T}^d} |\nabla y|^2/ C  + C |y|^2 \dd x$. Then $B(t)$ is semi-coercive, uniformly in time. Thus by \cite[Chapter 3, Theorem 1.2]{lions1971optimal} we have
\begin{align} \nonumber
 \| \delta \bar{m}_k\|_{L^{2}(0,T;V)} + \|\partial_t  \delta \bar{m}_k\|_{L^{2}(0,T;V^{*})} & \leq C \| f \|_{L^2(0,T;V^*)} \\ \nonumber
& \leq C \| \nabla \cdot \delta \bar{v}_k \bar{m}_k \|_{L^{2}(0,T;V^{*})} \\
&  \leq C \|\delta \bar{v}_k \bar{m}_k \|_{L^{2}(Q;\mathbb{R}^d)} \leq  C \epsilon^{1/2}_k. \nonumber
\end{align}
We conclude Step 3 with the continuous inclusion (see  \cite[Chapter 3, Theorem 1.1]{lions1971optimal})
\begin{equation} \nonumber
\{m \in L^2(0,T;V); \; \partial_t m \in L^2(0,T;V^*) \} \subseteq C(0,T;L^2(\mathbb{T}^d)).
\end{equation}

\noindent \textit{Step 4: $\| \delta \bar{w}_k \|_{L^{2}(Q;\mathbb{R}^d)}  \leq  C \epsilon^{1/2}_k$.}
By definition of $\delta  \bar{w}_k$ we have
\begin{align} \nonumber
\|\delta  \bar{w}_k\|_{L^2(Q;\mathbb{R}^d)}  \leq  \| \delta \bar{v}_k  \bar{m}_k\|_{L^2(Q;\mathbb{R}^d)} +  \|\bar{v} \delta \bar{m}_k \|_{L^2(Q;\mathbb{R}^d)}  \leq C \epsilon^{1/2}_k,
\end{align}
where the last inequality follows from Step 1 and Step 3.

\noindent \textit{Step 5: $\|\delta P_k  \|_{L^2(0,T;\mathbb{R}^k)} \leq  C \epsilon^{1/2}_k$ and $\| \delta \gamma_k \|_{L^\infty(Q)}\leq  C \epsilon^{1/2}_k$.}
Using that $\phi$ is Lipschitz with respect to its second variable (see Assumption \eqref{ass_Lip_phi}),
\begin{align} \nonumber
|\delta P_k(t) | = |\phi(t, A \bar{w}_k(t)) - \phi(t, A \bar{w}(t)) | & \leq C | A \delta  \bar{w}_k(t)|  \nonumber
\end{align}
for almost every $t\in [0,T]$. Since
\begin{equation} \nonumber
|A \delta  \bar{w}_k(t)| = \left| \int_{\mathbb{T}^d} a(x,t) \delta  \bar{w}_k(x,t) \dd x \right| \leq \|a(t) \|_{L^\infty(\mathbb{T}^d;\mathbb{R}^{k \times d})}  \| \delta \bar{w}_k(t) \|_{L^{1}(\mathbb{T}^d;\mathbb{R}^d)},
\end{equation}
Since $\|a\|_{L^\infty(Q;\mathbb{R}^{k\times d})} \leq C$, Step 4 yields the desired estimate
\begin{equation} \nonumber
\|\delta P_k \|_{L^2(0,T;\mathbb{R}^k)} \leq C \| \delta \bar{w}_k\|_{L^2(Q;\mathbb{R}^d)} \leq C \epsilon^{1/2}_k.
\end{equation}
Using that $f$ is Lipschitz with respect to its third variable (see Assumption \eqref{ass_hold_f}) yields
\begin{equation} \nonumber
|\delta \gamma_k(x,t)| = | f(x,t,\bar{m}_k(t)) - f(x,t,\bar{m}(t)) | \leq C \| \delta \bar{m}_k(t) \|_{L^{2}(\mathbb{T}^d)},
\end{equation}
for any $(x,t) \in Q$. Taking the supremum over $(x,t) \in Q$ both sides of the inequality yields that $\|\delta \gamma_k \|_{C(Q)} \leq C \epsilon^{1/2}_k$ by Step 3, which concludes the step.

\noindent \textit{Step 6: $\|\delta u_k \|_{C(Q)} \leq  C \epsilon^{1/2}_k$.} 
Since $\|\gamma_k\|_{L^{\infty}(Q)} \leq C$, $\|P_k\|_{L^2(0,T;\mathbb{R}^k)} \leq C$, $\|\bar{\gamma}\|_{L^{\infty}(Q)} \leq C$, and $\|\bar{P}\|_{L^2(0,T;\mathbb{R}^k)} \leq C$, Lemma \ref{lem:u-L-infty} yields
\begin{equation} \nonumber
\delta u_k(x,t)  =  \inf_{\alpha \in L_{\mathbb{F}}^{2,C}(t,T;\mathbb{R}^d)} J_{\gamma_k,P_k}(x,t,\alpha) -  \inf_{\alpha' \in L_{\mathbb{F}}^{2,C}(t,T;\mathbb{R}^d)} J_{\bar{\gamma},\bar{P}}(x,t,\alpha'),
\end{equation}
for any $(x,t) \in Q$.
We denote $(X^\alpha_{s})_{s \in [t,T]}$ the solution to the stochastic differential equation $\dd X_s = \alpha_s \dd s + \sqrt{2} \dd B_s$ with $X^\alpha_t = x$, for any $\alpha \in L_{\mathbb{F}}^{2}(t,T;\mathbb{R}^d)$. Then
\begin{align}  \nonumber
|\delta u_k(x,t)|  & \leq \sup_{\alpha \in L_{\mathbb{F}}^{2,C}(t,T;\mathbb{R}^d)} | J_{\gamma_k,P_k}(x,t,\alpha) - J_{\bar{\gamma},\bar{P}}(x,t,\alpha)| \\
& \leq \sup_{\alpha \in L_{\mathbb{F}}^{2,C}(t,T;\mathbb{R}^d)} \mathbb{E} \left[ \int_t^T  |\langle A^\star [\delta P_k](X^\alpha_s,s) , \alpha_s \rangle| + |\delta \gamma_k(X^\alpha_s,s)|  \dd s \right].  \nonumber
\end{align}
For any $(x,s) \in Q$ and $\alpha \in \mathbb{R}^d$, the Cauchy-Schwarz inequality  yields
\begin{align} \nonumber
|\langle A^\star [\delta P_k](x,s) , \alpha \rangle | & \leq |\langle a(x,s) \delta P_k(s)| |\alpha | \\ &
\leq \|a\|_{L^\infty(Q;\mathbb{R}^{k \times d})} | \delta P_k(s)| |\alpha |.
\nonumber
\end{align}
Since $\|a\|_{L^\infty(Q;\mathbb{R}^{k \times d})} \leq C$, we finally have
\begin{equation}  \nonumber
|\delta u_k(x,t)| \leq C \left(\| \delta P_k \|_{L^2(0,T;\mathbb{R}^k)}  + \|\delta \gamma_k \|_{L^\infty(Q)} \right).
\end{equation}
Thus Step 6 holds by Step 5, which concludes the proof.
\end{proof}

Let us comment our last convergence results: Lemma \ref{lemma:exploitability} and Theorem \ref{lemma:conv-seq}. For the fictitious play learning rate $\delta_k = 1/(k+1)$, we have proved that the primal gap sequence $(\epsilon_k)_{k\in\mathbb{N}}$ converges in $O(\ln(k)/k)$ and the exploitability sequence $(\sigma_k)_{k\in\mathbb{N}}$ and the sequence of variables $(\bar{m}_k,\bar{w}_k,\bar{v}_k,P_k,\gamma_k,u_k)_{k\in\mathbb{N}}$ converge in $O(\sqrt{\ln(k)/k})$. We have obtained a sharper convergence result for the Frank-Wolfe learning rate $\delta_k = 2/(k+2)$. For this choice, we have shown that the primal gap sequence $(\epsilon_k)_{k\in\mathbb{N}}$ converges in $O(1/k)$ and the exploitability sequence $(\sigma_k)_{k\in\mathbb{N}}$ and the sequence of variables $(\bar{m}_k,\bar{w}_k,\bar{v}_k,P_k,\gamma_k,u_k)_{k\in\mathbb{N}}$ converge in $O(\sqrt{1/k})$. The convergence results for $(\bar{m}_k,\bar{w}_k,\bar{v}_k,P_k,\gamma_k,u_k)$ are new in the mean field game literature.

We conclude this section with a discussion on our results. The results concerning the convergence of the primal gap and the exploitability (Lemmas \ref{lemma:rate-conv} and \ref{lemma:exploitability}) are the same as those obtained in \cite{geist2021concave} for different mean field game models, with a discrete structure.
These results are indeed general, since they only rely on the convexity structure of the potential problem and the regularity properties of the coupling terms.
Therefore, they could certainly be adapted to other models, for example first order mean field games.

We also expect that similar convergence results, for the coupling terms, the value function, and the distribution, could be obtained in a different framework. A key step in the proof would be to establish a quadratic growth property (as the one obtained in Lemma \ref{lemma:J-J}), under a strong convexity assumption on the running cost $L$.


\appendix

\section{Regularity of the Hamiltonian \label{appendix:reg-H}}

Some properties of the Hamiltonian can be deduced from the convexity assumption \eqref{eq:grad_monotony} and the H\"older continuity of $L$ and its derivatives (Assumption \eqref{ass_holder}). They are collected in the following lemmas. whose proofs can be found in \cite{BHP-schauder}.

\begin{lemma} \label{lemma:reg_H}
The Hamiltonian $H$ is differentiable with respect to $p$ and $H_p$ is differentiable with respect to $x$ and $p$. Moreover, for all $R>0$, there exists $\alpha \in (0,1)$ such that $H \in \mathcal{C}^{\alpha}(B_R)$, $H_p \in \mathcal{C}^{\alpha}(B_R, \R^d)$, $H_{px} \in \mathcal{C}^{\alpha}(B_R, \R^{d \times d})$, and $H_{pp} \in \mathcal{C}^{\alpha}(B_R, \R^{d \times d})$.
\end{lemma}

\begin{proof}
See \cite[Lemma 1]{BHP-schauder}.
\end{proof}

\begin{lemma} \label{lemma:H-L-quad}
There exists a constant $C > 0$ such that for all $(x,t) \in Q$, for all
$p\in \mathbb{R}^d$ and for all $v\in \mathbb{R}^d$,
\begin{equation} \label{eq:H+L}
H(x,t,p) + L(x,t,v) + \langle p,v \rangle \geq \frac{1}{C}|v + H_p(x,t,p)|^2.
\end{equation}
In addition for any $m,\bar{m} \geq 0$ and $\bar{v} = - H_p(x,t,p)$ we have that
\begin{equation}
L(x,t,v) m - L(x,t,\bar{v}) \bar{m} \geq  - H(x,t,p) (m - \bar{m}) 
- \langle p,w - \bar{w} \rangle + \frac{1}{C}|v - \bar{v}|^2 m, \label{ineq-conv:L-L}
\end{equation}
where $(w,\bar{w}) \coloneqq (mv,\bar{m}\bar{v})$.
\end{lemma}

\begin{proof}
See \cite[Proof of Proposition 2]{BHP-schauder}.
\end{proof}

\section{A priori bounds for parabolic equations}

In this appendix we provide estimates for the following parabolic equation:
\begin{equation}
  \label{eq:parabolic}
\begin{array}{rlr}
\partial_t u - \sigma \Delta u + \langle b, \nabla u \rangle+ c u = & h, \quad & (x,t) \in Q, \\
u(x,0)= & u_0(x), & x \in \mathbb{T}^d,
\end{array}
\end{equation}
for different assumptions on $b$, $c$, $h$, and $u_0$. The proofs of the following results can be found in the Appendix of \cite{BHP-schauder}; they largely rely on \cite{LSU}. We recall that $p$ is a fixed parameter and $p>d+2$. 

\begin{theorem} 
\label{theo:max_reg1}
 For all $R>0$, there exists $C>0$ such that for all $u_0 \in W^{p,2-2/p}(\mathbb{T}^d)$, for all $b \in L^p(Q;\R^d)$, for all $c \in L^p(Q)$, for all $h \in L^p(Q)$, satisfying
\begin{equation*}
\| u_0 \|_{W^{p,2-2/p}(\mathbb{T}^d)} \leq R, \quad
\| b \|_{L^p(Q;\R^d)} \leq R, \quad
\| c \|_{L^p(Q)} \leq R, \quad
\| h \|_{L^p(Q)} \leq R,
\end{equation*}
equation \eqref{eq:parabolic} has a unique solution $u$ in $W^{2,1,p}(Q)$. Moreover, $\| u \|_{W^{2,1,p}(Q)} \leq C$.
\end{theorem}

\begin{theorem}
\label{LAD-l3.4p82}
For $q \in (1,\infty)$, the trace at time $t=0$ of elements of $W^{2,1,q}(Q)$
belongs to $W^{q,2-2/q}(\Omega)$. 
\end{theorem}

\begin{theorem} \label{thm:bound-u-u0-h}
There exists $C>0$ such that for all $u_0 \in W^{2-2/p,p}(\mathbb{T}^d)$ and for all $h \in L^p(Q)$, the unique solution $u$ to \eqref{eq:parabolic} (with $b = 0$ and $c=0$) satisfies the following estimate:
\begin{equation} \nonumber
\|u\|_{W^{2,1,p}(Q)} \leq C \left( \|u_0\|_{W^{2-2/p,p}(\mathbb{T}^d)} + \|h\|_{L^p(Q)} \right). 
\end{equation}
\end{theorem}

\begin{lemma}
  \label{lemma:max_reg_embedding}
There exists $\delta \in (0,1)$ and $C>0$ such that for all $u \in
W^{2,1,p}(Q)$,
\begin{equation*}
\| u \|_{\mathcal{C}^\delta(Q)} + \| \nabla u \|_{\mathcal{C}^\delta(Q;\R^d)}
\leq C \| u \|_{W^{2,1,p}(Q)}.
\end{equation*}
\end{lemma}

\begin{theorem}
    \label{theo:holder_reg_classical}
    For all $\beta \in (0,1)$, for all $R>0$, there exist $\alpha \in (0,1)$ and $C>0$ such that for all $u_0 \in \mathcal{C}^{2+ \beta}(\mathbb{T}^d)$, $b \in \mathcal{C}^{\beta,\beta/2}(Q;\R^d)$, $c \in \mathcal{C}^{\beta,\beta/2}(Q)$ and $h \in \mathcal{C}^{\beta,\beta/2}(Q)$ satisfying
\begin{equation*}
\| u_0 \|_{\mathcal{C}^{2+ \beta}(\mathbb{T}^d)} \leq R, \ \
\| b \|_{\mathcal{C}^{\beta,\beta/2}(Q;\R^d)} \leq R, \ \
\| c \|_{\mathcal{C}^{\beta,\beta/2}(Q)} \leq R, \ \ \text{and} \ \
\| h \|_{\mathcal{C}^{\beta,\beta/2}(Q)} \leq R,
\end{equation*}
the solution to \eqref{eq:parabolic} lies in $\mathcal{C}^{2+\alpha,1+\alpha/2}(Q)$ and satisfies
$\| u \|_{\mathcal{C}^{2+\alpha,1+\alpha/2}(Q)} \leq C$.
\end{theorem}

\section{Maximum principle \label{appendix:maximum-principle}}

In this appendix we establish a maximum principle for the Fokker-Planck equation. We study the parabolic equation \eqref{eq:parabolic} with $h=0$,
\begin{equation}
\label{eq:fokkerplanck}
\begin{array}{rlr}
\partial_t m - \sigma \Delta m + \langle b, \nabla m \rangle+ c m = & 0, \quad & (x,t) \in Q, \\
m(x,0)= & m_0(x), & x \in \mathbb{T}^d.
\end{array}
\end{equation}
We assume that $m_0$ satisfies Assumption \eqref{ass_init_cond}
and define the mapping $\bar{\mathsf{m}} \colon L^\infty(Q; \mathbb{R}^d) \times L^\infty(Q)  \to W^{2,1,p}(Q)$ which associates to any $(b,c)$ the solution to \eqref{eq:fokkerplanck}. By Theorem \ref{theo:max_reg1} the mapping $\bar{\mathsf{m}}$ is well-defined.

\begin{lemma} \label{lemma:continuity-m-b-c}
The mapping $\bar{\mathsf{m}}$ is continuous.
\end{lemma}

\begin{proof}
Consider the mapping $\varphi : W^{2,1,p}(Q) \times L^\infty(Q; \mathbb{R}^d) \times L^\infty(Q) \to W^{p,2-2/p}(\mathbb{T}^d) \times L^p(Q)$ defined by
\begin{equation} \nonumber
\varphi[m,b,c] = (m(0,\cdot) - m_0(\cdot), \partial_t m - \sigma \Delta m + \langle b, \nabla m \rangle+ c m).
\end{equation}
We define
\begin{equation} \nonumber
\begin{array}{llll}
\varphi_0[m] & = m(0,\cdot), & \varphi_2[m,b] & = \langle b, \nabla m \rangle,  \\
\varphi_1[m] & =  \partial_t m - \sigma \Delta m, \qquad & \varphi_3[m,c] & = cm,
\end{array}
\end{equation}
so that $\varphi[m,b,c] = (\varphi_0[m] - m_0(\cdot),  \varphi_1[m] + \varphi_2[m,b] + \varphi_3[m,c])$. By Theorem \ref{theo:max_reg1} and Theorem \ref{LAD-l3.4p82}, there exists a constant $C >0$ such that
\begin{equation} \nonumber
\begin{array}{llll}
 \| \varphi_0[m] \|_{W^{p,2-2/p}(\mathbb{T}^d)} \! \! & \leq C \| m \|_{W^{2,1,p}(Q)}, & \| \varphi_2[m,b]  \|_{L^p(Q)} \! \! \! \! & \leq \|b\|_{L^{\infty}(Q)} \| m \|_{W^{2,1,p}(Q)},\\
\|\varphi_1[m] \|_{L^p(Q)} & \leq C \| m \|_{W^{2,1,p}(Q)}, \quad & \| \varphi_3[m,c]  \|_{L^p(Q)} \! \! \! \! & \leq \|c\|_{L^{\infty}(Q)}  \| m \|_{W^{2,1,p}(Q)}.
\end{array}
\end{equation}
Thus $\varphi_0$ and $\varphi_1$ (resp.\@ $\varphi_2$ and $\varphi_3$) are  $C^\infty$ as bounded linear (resp.\@ bi-linear) applications. It follows that $\varphi$ is $C^\infty$.
Let $(m,b,c) \in W^{2,1,p}(Q) \times L^p(Q; \mathbb{R}^d) \times L^p(Q)$ be such that $\varphi[m,b,c] = 0$. For any direction $z \in W^{2,1,p}(Q)$, we have
\begin{equation} \nonumber
D_m \varphi [m,b,c] z  = (z(0,\cdot), \partial_t z - \sigma \Delta z + \langle b, \nabla z \rangle + c z).
\end{equation}
For any $(h_0,h_1) \in  W^{p,2-2/p}(\mathbb{T}^d) \times L^p(Q)$, the equation
\begin{equation}\nonumber
\begin{array}{rlr}
\partial_t z - \sigma \Delta z + \langle b, \nabla z \rangle + c z = & h_1, \quad & (x,t) \in Q, \\
z(x,0)= & h_0, & x \in \mathbb{T}^d,
\end{array}
\end{equation} 
has a unique solution $z \in W^{2,1,p}(Q)$, by Theorem \ref{theo:max_reg1}. Then $D_m \varphi [m,b,c]$ is bijective and thus invertible. The conclusion follows by the implicit function theorem.
\end{proof}

\begin{lemma} \label{lemma:maximum-principle}
Let $v \in W^{1,0,\infty}(Q;\mathbb{R}^d)$ and let 
$m = \bar{\mathsf{m}}[v,\nabla \cdot v] \in W^{2,1,p}(Q)$ be the solution to \eqref{eq:fokkerplanck} with $(b,c) = (v, \nabla \cdot v)$.
Assume that $m_0(x) \geq \varepsilon_0>0$ for any $x \in \mathbb{T}^d$.
Then
\begin{equation}  \label{eq:maximum-principle}
m(x,t) \geq \varepsilon_0  \exp \big( - T \| \nabla \cdot v \|_{L^\infty(Q)} \big), \quad \forall (x,t) \in Q.
\end{equation}
\end{lemma}

\begin{proof}
We first prove the result when $v \in \mathcal{C}^{1+\alpha,\alpha/2}(Q;\R^d)$, for some $\alpha \in (0,1)$.
By Theorem \ref{theo:holder_reg_classical}, $m \in C^{2,1}(Q)$. Let $\kappa > \| \nabla \cdot v \|_{L^\infty(Q)}$.
We define
\begin{equation}\nonumber
y(x,t) = e^{-\kappa t} \left(m(x,t)- \varepsilon_0  e^{-\kappa t} \right),
\end{equation}
for all $(x,t) \in Q$. By a direct computation we have
\begin{align}  \nonumber
\partial_t y(x,t) = - y(x,t)(\kappa - \nabla \cdot v (x,t)) + \Delta y(x,t) + \langle v(x,t), \nabla y(x,t)\rangle \\ + \varepsilon_0 e^{- 2\kappa  t} \left(\kappa +  \nabla \cdot v (x,t) \right). \label{partial-y}
\end{align}
Next we show that $y(x,t) \geq 0$ for all $(x,t) \in Q$.
Let $(x_0,t_0) \in \argmin_{(x,t) \in Q} y(x,t)$. Let us assume, by a way of contradiction, that $y(x_0,t_0) < 0$. Since $y(0,x) \geq 0$ for any $x \in \mathbb{T}^d$, we have that $t_0 > 0$ and thus $\partial_t y(x_0,t_0) \leq 0$. Since $x_0 \in \mathbb{T}^d$, we have that $\nabla y(x_0,t_0) = 0$. Moreover, since $m$ is twice differentiable with respect to its second variable, we have that $\Delta y (x_0,t_0) \geq 0$. Then it follows from \eqref{partial-y} that
\begin{equation} \nonumber
\partial_t y(x_0,t_0) \geq - y(x_0,t_0) (\kappa  - \nabla \cdot v (x_0,t_0)) \\ + \varepsilon_0 e^{-2 \kappa  t_0} \left(\kappa +  \nabla \cdot v (x_0,t_0) \right).
\end{equation}
The right-hand side is positive since $\kappa > \| \nabla \cdot v \|_{L^\infty(Q)}$. This contradicts the inequality $\partial_t y(x_0,t_0) \leq 0$ and proves that $y(x,t) \geq 0$, for any $(x,t) \in Q$. It follows then from the definition of $y$ that
$m(x,t) \geq \varepsilon_0  e^{-\kappa t} $, for any $(x,t) \in Q$. Passing to the limit when $\kappa \rightarrow $ yields \eqref{eq:maximum-principle}.

We now consider the general case when $v \in W^{1,0,\infty}(Q;\mathbb{R}^d)$ and proceed by density. Let $(\rho_k)_{k \in \mathbb{N}}$ be a sequence of regularizing kernels in $C^\infty(Q)$. We define $v_k = \rho_k * v \in C^\infty(Q;\mathbb{R}^d)$, where $*$ is the convolution product. We next define $m_k = \bar{\mathsf{m}}[v_k,\nabla \cdot v_k]$ for any $k\in \mathbb{N}$.
Applying \eqref{eq:maximum-principle} to $m_k$, we obtain that
\begin{equation*}
m_k(x,t) \geq \varepsilon_0  \exp \big( - T \| \nabla \cdot v_k \|_{L^\infty(Q)} \big), \quad \forall (x,t) \in Q.
\end{equation*}
Since $v_k$ (resp.\@ $\nabla v_k$) uniformly converges to $v$ (resp.\@ $\nabla v$) and since $\bar{\mathsf{m}}$ is continuous for the uniform topology, we deduce that $m_k$ converges to $m$ in $W^{2,1,p}(Q)$ and finally that $m_k$ uniformly converges to $m$, by Lemma \ref{lemma:max_reg_embedding}. This allows us to pass to the limit in the above inequality, which concludes the proof of the lemma.
\end{proof}

\section{Existence of a classical solution to the Hamilton-Jacobi-Bellman equation \label{Appendix:HJB}}

In this appendix we prove Lemma \ref{eq:HJB-unique-solution}, that is, we establish the existence of a solution to the Hamilton-Jacobi-Bellman equation
\begin{equation} \label{eq:HJB-appendix}
\begin{array}{rlr}
- \partial_t u - \Delta u + \bm{H}[  \nabla u + A^\star P] = & \gamma, & (x,t) \in Q, \\
 u(x,T) = &  g(x), & x\in \mathbb{T}^d,
\end{array}
\end{equation}
in $C^{2,1}(Q)$. By classical, we mean that \eqref{eq:HJB-appendix} can be understood in a pointwise manner.
We recall that $g \in \mathcal{C}^{2 +\alpha}(\mathbb{T}^d)$ and that $(\gamma,P) \in \mathcal{U}^\beta$ (defined in \eqref{eq:setUbeta}). 
Moreover, the constant $R >0$ is such that
\begin{equation} \label{appendix:estim-ugp}
 \|\gamma\|_{L^\infty(Q)} + \|\nabla \gamma\|_{L^\infty(Q;\mathbb{R}^d)} + \| P \|_{L^\infty(0,T;\mathbb{R}^k)} \leq R.
\end{equation}

The proof of the lemma relies on a fixed point approach. To this purpose, we introduce the mapping $\mathcal{T} \colon W^{2,1,p}(Q) \times [0,1] \to W^{2,1,p}(Q)$ which associates to any $u \in W^{2,1,p}(Q)$ the classical solution $\mathcal{T}[u,\tau]$ to the linear parabolic equation
\begin{equation} \nonumber
\begin{array}{rlr}
- \partial_t \tilde{u} - \Delta \tilde{u} + \tau \bm{H}[\nabla u + A^\star P] = & \tau \gamma & (x,t) \in Q, \\
 \tilde{u}(x,T) = &  \tau g(x) & x\in \mathbb{T}^d.
\end{array}
\end{equation}
For any $(u,\tau) \in W^{2,1,p}(Q) \times [0,1]$, we have $\tau \big( \gamma - \bm{H}[\nabla u + A^\star P] \big) \in L^\infty(Q)$, thus $\mathcal{T}[u,\tau]$ lies in $W^{2,1,p}(Q)$, by Theorem \ref{theo:max_reg1}, proving that $\mathcal{T}$ is well-defined.

\begin{lemma} \label{lemma:cont-comp-T}
 The mapping $\mathcal{T} \colon W^{2,1,p}(Q) \times [0,1] \to W^{2,1,p}(Q)$ is continuous and compact. In addition, for all $K >0$, there exists $\alpha \in (0,1)$ and $C>0$ depending on $K$, $\gamma$, and $P$ such that $\|u\|_{W^{2,1,p}(Q)} \leq K$ implies $\|\mathcal{T}[u,\tau]\|_{\mathcal{C}^{2+\alpha,1+ \alpha/2}(Q)} \leq C$.
\end{lemma}

\begin{proof}
\noindent \textit{Step 1: Continuity of $\mathcal{T}$.}
Let $(u_k,\tau_k) \in W^{2,1,p}(Q) \times [0,1]$ be a sequence converging to $(u,\tau) \in W^{2,1,p}(Q) \times [0,1]$. Then $\nabla u_k \to \nabla u$ in $L^\infty(Q;\mathbb{R}^d)$ by Lemma \ref{lemma:max_reg_embedding}. Then $\tau_k (\gamma - \bm{H}[\nabla u_k + A^\star P]) \to \tau (\gamma - \bm{H}[\nabla u + A^\star P])$ in $L^\infty(Q;\mathbb{R}^d)$ by continuity of the Hamiltonian (see Lemma \ref{lemma:reg_H}). Finally the continuity of $\mathcal{T}$ follows by Theorem \ref{thm:bound-u-u0-h}.

\noindent \textit{Step 2: Compactness of $\mathcal{T}$.} Let $K>0$ and let $(u,\tau) \in W^{2,1,p}(Q) \times [0,1]$ be such that $\|u\|_{W^{2,1,p}(Q)} + |\tau| < K$. 
Combining Lemma \ref{lemma:max_reg_embedding} and Lemma \ref{lemma:reg_H} there exist $\alpha \in (0,1)$ and $C>0$ such that $ \|\tau(\gamma - \bm{H}[\nabla u + A^\star P])\|_{C^\alpha(Q)} < C$. Then applying Theorem \ref{theo:holder_reg_classical}, there exist $\alpha \in (0,1)$ and $C>0$ such that $\|\mathcal{T}[u,\tau]\|_{\mathcal{C}^{2+\alpha,1+ \alpha/2}(Q)} < C$. By the Arzela-Ascoli Theorem the centered ball of $\mathcal{C}^{2+\alpha,1+ \alpha/2}(Q)$ of radius $C>0$ is a relatively compact subset of $W^{2,1,p}(Q)$. As a consequence $\mathcal{T}[u,\tau]$ is a compact mapping and the conclusion follows.
\end{proof}

\begin{theorem}{\normalfont (Leray-Schauder)} \label{thm:Leray-schauder}
Let $X$ be a Banach space and let $T : X \times [0, 1] \to X$ be a continuous and compact mapping. Assume that $T (x,0) = 0$ for all $x\in X$ and assume there exists $C>0$ such that $\|x\|_X < C$ for all $(x,\tau) \in X\times [0,1]$ such that $T (x,\tau) = x$. Then, there exists $x \in X$ such that $T(x,1) = x$.
\end{theorem}

\begin{proof}
See \cite[Theorem 11.6]{gilbarg2015elliptic}.
\end{proof}

\begin{proof}[Proof of Lemma \ref{eq:HJB-unique-solution}]
\noindent \textit{Step 1: Existence of a classical solution.}
 We have that $\mathcal{T}[u,0] = 0$ for all $u \in W^{2,1,p}(Q)$. Now let $(u,\tau) \in W^{2,1,p}(Q) \times [0,1]$ such that $\mathcal{T}[u,\tau] = u$. From Lemma \ref{lemma:cont-comp-T}, the mapping $\mathcal{T}$ is continuous and compact, in addition $u$ is a classical solution and thus the viscosity solution to the Hamilton-Jacobi-Bellman equation
\begin{equation} \nonumber
\begin{array}{rlr}
- \partial_t u - \Delta u + \tau \bm{H}[\nabla u + A^\star P] = & \tau \gamma & (x,t) \in Q, \\
 u(x,T) = &  \tau g(x) & x\in \mathbb{T}^d,
\end{array}
\end{equation}
and can be interpreted as the value function associated to the following stochastic control problem
\begin{equation} \nonumber
\inf_{\nu \in L_{\mathbb{F}}^2(0,T;\mathbb{R}^d)} 
 \tau \mathbb{E} \left[\int_0^T  L(X^\tau_s,s,\nu_s) + \langle A^\star [P](X^\tau_s,s) , \nu_s \rangle + \gamma(X^\tau_s,s) \dd s + g(X^\tau_T) \right],
\end{equation}
where $(X^\tau_{s})_{s \in [t,T]}$ is the solution to $\dd X_s = \tau \nu_s \dd s + \sqrt{2} \dd B_s$, $X_0 = Y$. Following \cite[Proposition 1, Step 2]{BHP-schauder}, there exists a constant $C>0$, depending only on $R$, such that $\|u\|_{L^\infty(Q)} + \|\nabla u\|_{L^\infty(Q;\mathbb{R}^d)} \leq C$.
Assumption \eqref{appendix:estim-ugp} yields that $\| \bm{H}[\nabla u + A^\star P] - \gamma\|_{L^\infty(Q)} \leq C$. Then $u$ is the solution to a parabolic PDE with bounded coefficients and thus $\| u \|_{W^{2,1,p}(Q)} \leq C$, by Theorem \ref{theo:max_reg1}. Again, $C$ only depends on $R$.
By the Leray-Schauder Theorem \ref{thm:Leray-schauder}, there exists a classical solution to \eqref{eq:HJB-appendix}.

\noindent \textit{Step 2: Uniqueness.} Let $u_1$ and $u_2$ be two classical solutions to \eqref{eq:HJB-appendix}. Then $u_1$ and $u_2$ are viscosity solutions to \eqref{eq:HJB-appendix}.
Since the viscosity solution is unique, $u_1=u_2$.

\noindent \textit{Step 3: $\|u\|_{W^{2,1,p}(Q)} + \|\nabla u\|_{W^{2,1,p}(Q)} \leq C$.}
We have already obtained a bound on $\| u \|_{W^{2,1,p}(Q)}$ in Step 1.
It remains to show that $\|\nabla u\|_{W^{2,1,p}(Q)} \leq C$.
Let $i \in \left\{1,\ldots,d \right\}$. We have that $u^i \coloneqq \partial_{x_i} u$ is the solution to the following equation
\begin{equation}  \nonumber
\begin{array}{rl}
- \partial_t  u^i  - \Delta u^i + \partial_{x_i} \bm{H}[ \nabla u + A^\star P] +  \bm{H}_p[ \nabla u + A^\star P] \cdot (\nabla u^i + \partial_{x_i}A^\star P)  = & \partial_{x_i} \gamma, \\
 u^i(T)  = &  \partial_{x_i} g,
\end{array}
\end{equation}
for any $(x,t) \in Q$. By Lemma \ref{lemma:reg_H}, $\partial_{x_i} \bm{H}$ and $\bm{H}_p$ are continuous, thus 
\begin{equation} \nonumber
 \| \partial_{x_i} \bm{H}[\nabla u + A^\star P]\|_{L^\infty(Q)} \leq C, \quad \| \bm{H}_p[\nabla u + A^\star P]\|_{L^\infty(Q;\mathbb{R}^d)} \leq C,
\end{equation}
since $\|\nabla u\|_{\mathcal{C}^\alpha(Q;\mathbb{R}^d)} \leq C$ and $\|A^\star P\|_{L^\infty(Q;\mathbb{R}^d)} \leq C$.
By Assumption \eqref{ass_holder}, $\partial_{x_i} a$ is continuous, therefore $\|\partial_{x_i} A^\star P \|_{L^\infty(Q;\mathbb{R}^d)} \leq C$. We further have $\|\nabla \gamma_k\|_{L^\infty(Q;\mathbb{R}^d)} \leq C$ and $ \|\partial_{x_i} g\|_{L^\infty(\mathbb{T}^d)} \leq C$, by Assumption \eqref{ass_init_cond}. It follows that $u^i_k$ is the solution of a parabolic PDE with bounded coefficients, thus by Theorem \ref{theo:max_reg1}, $\|u^i\|_{W^{2,1,p}(Q)} \leq C$ and the Step 3 is proved which concludes the proof.
\end{proof}

\bibliographystyle{plain}
\bibliography{biblio}

\begin{thebibliography}{10}

\bibitem{achdou2014partial}
Yves Achdou, Francisco~J Buera, Jean-Michel Lasry, Pierre-Louis Lions, and
  Benjamin Moll.
\newblock Partial differential equation models in macroeconomics.
\newblock {\em Philosophical Transactions of the Royal Society A: Mathematical,
  Physical and Engineering Sciences}, 372(2028):20130397, 2014.

\bibitem{kobeissi2020meanfinite}
Yves Achdou and Ziad Kobeissi.
\newblock Mean field games of controls: Finite difference approximations, 2020.

\bibitem{achdou2020mean}
Yves Achdou and Mathieu Lauri{\`e}re.
\newblock Mean field games and applications: Numerical aspects.
\newblock {\em arXiv preprint arXiv:2003.04444}, 2020.

\bibitem{alasseur2020extended}
Cl{\'e}mence Alasseur, Imen~Ben Tahar, and Anis Matoussi.
\newblock An extended mean field game for storage in smart grids.
\newblock {\em Journal of Optimization Theory and Applications},
  184(2):644--670, 2020.

\bibitem{bauso2016opinion}
Dario Bauso, Hamidou Tembine, and Tamer Basar.
\newblock Opinion dynamics in social networks through mean-field games.
\newblock {\em SIAM Journal on Control and Optimization}, 54(6):3225--3257,
  2016.

\bibitem{Benamou2015}
Jean-David Benamou and Guillaume Carlier.
\newblock Augmented {L}agrangian methods for transport optimization, mean field
  games and degenerate elliptic equations.
\newblock {\em Journal of Optimization Theory and Applications}, 167(1):1--26,
  Oct 2015.

\bibitem{benamou2019entropy}
Jean-David Benamou, Guillaume Carlier, Simone Di~Marino, and Luca Nenna.
\newblock An entropy minimization approach to second-order variational
  mean-field games.
\newblock {\em Mathematical Models and Methods in Applied Sciences},
  29(08):1553--1583, 2019.

\bibitem{benamou2017variational}
Jean-David Benamou, Guillaume Carlier, and Filippo Santambrogio.
\newblock Variational mean field games.
\newblock In {\em Active Particles, Volume 1}, pages 141--171. Springer, 2017.

\bibitem{BHP-schauder}
J.~Fr{\'e}d{\'e}ric Bonnans, Saeed Hadikhanloo, and Laurent Pfeiffer.
\newblock Schauder estimates for a class of potential mean field games of
  controls.
\newblock {\em Applied Mathematics \& Optimization}, 83:1431--1464, 2021.

\bibitem{bonnans2021discrete}
J.~Fr{\'e}d{\'e}ric Bonnans, Pierre Lavigne, and Laurent Pfeiffer.
\newblock Discrete potential mean field games.
\newblock 2021.

\bibitem{bredies2009generalized}
Kristian Bredies, Dirk~A Lorenz, and Peter Maass.
\newblock A generalized conditional gradient method and its connection to an
  iterative shrinkage method.
\newblock {\em Computational Optimization and Applications}, 42(2):173--193,
  2009.

\bibitem{briceno2019implementation}
Luis Brice{\~n}o-Arias, Dante Kalise, Ziad Kobeissi, Mathieu Lauri{\`e}re,
  A.~Mateos Gonz{\'a}lez, and Francisco~J. Silva.
\newblock On the implementation of a primal-dual algorithm for second order
  time-dependent mean field games with local couplings.
\newblock {\em ESAIM: Proceedings and Surveys}, 65:330--348, 2019.

\bibitem{briceno2018proximal}
Luis~M. Briceno-Arias, Dante Kalise, and Francisco~J. Silva.
\newblock Proximal methods for stationary mean field games with local
  couplings.
\newblock {\em SIAM Journal on Control and Optimization}, 56(2):801--836, 2018.

\bibitem{brown1951iterative}
George~W. Brown.
\newblock Iterative solution of games by fictitious play.
\newblock {\em Activity analysis of production and allocation}, 13(1):374--376,
  1951.

\bibitem{cardaliaguet2015second}
Pierre Cardaliaguet, P.~Jameson Graber, Alessio Porretta, and Daniela Tonon.
\newblock Second order mean field games with degenerate diffusion and local
  coupling.
\newblock {\em Nonlinear Differential Equations and Applications NoDEA},
  22(5):1287--1317, 2015.

\bibitem{cardaliaguet2017learning}
Pierre Cardaliaguet and Saeed Hadikhanloo.
\newblock Learning in mean field games: the fictitious play.
\newblock {\em ESAIM: Control, Optimisation and Calculus of Variations},
  23(2):569--591, 2017.

\bibitem{cardaliaguet2016mean}
Pierre Cardaliaguet and Charles-Albert Lehalle.
\newblock Mean field game of controls and an application to trade crowding.
\newblock {\em Mathematics and Financial Economics}, 12(3):335--363, 2018.

\bibitem{cardaliaguet2016first}
Pierre Cardaliaguet, Alp{\'a}r~R. M{\'e}sz{\'a}ros, and Filippo Santambrogio.
\newblock First order mean field games with density constraints: pressure
  equals price.
\newblock {\em SIAM Journal on Control and Optimization}, 54(5):2672--2709,
  2016.

\bibitem{carmona2017mean}
Ren{\'e} Carmona, Fran{\c{c}}ois Delarue, and Daniel Lacker.
\newblock Mean field games of timing and models for bank runs.
\newblock {\em Applied Mathematics \& Optimization}, 76(1):217--260, 2017.

\bibitem{carmona2015probabilistic}
Ren{\'e} Carmona and Daniel Lacker.
\newblock A probabilistic weak formulation of mean field games and
  applications.
\newblock {\em The Annals of Applied Probability}, 25(3):1189--1231, 2015.

\bibitem{couillet2012electrical}
Romain Couillet, Samir~M. Perlaza, Hamidou Tembine, and M{\'e}rouane Debbah.
\newblock Electrical vehicles in the smart grid: A mean field game analysis.
\newblock {\em IEEE Journal on Selected Areas in Communications},
  30(6):1086--1096, 2012.

\bibitem{cui2021approximately}
Kai Cui and Heinz Koeppl.
\newblock Approximately solving mean field games via entropy-regularized deep
  reinforcement learning.
\newblock In {\em International Conference on Artificial Intelligence and
  Statistics}, pages 1909--1917. PMLR, 2021.

\bibitem{delarue2021exploration}
Fran{\c{c}}ois Delarue and Athanasios Vasileiadis.
\newblock Exploration noise for learning linear-quadratic mean field games.
\newblock {\em arXiv preprint arXiv:2107.00839}, 2021.

\bibitem{doncel2020mean}
Josu Doncel, Nicolas Gast, and Bruno Gaujal.
\newblock A mean field game analysis of {SIR} dynamics with vaccination.
\newblock {\em Probability in the Engineering and Informational Sciences},
  pages 1--18, 2020.

\bibitem{dunn1978conditional}
Joseph~C. Dunn and S.~Harshbarger.
\newblock Conditional gradient algorithms with open loop step size rules.
\newblock {\em Journal of Mathematical Analysis and Applications},
  62(2):432--444, 1978.

\bibitem{elie2020contact}
Romuald Elie, Emma Hubert, and Gabriel Turinici.
\newblock Contact rate epidemic control of covid-19: an equilibrium view.
\newblock {\em Mathematical Modelling of Natural Phenomena}, 15:35, 2020.

\bibitem{elie2019approximate}
Romuald Elie, Julien P{\'e}rolat, Mathieu Lauri{\`e}re, Matthieu Geist, and
  Olivier Pietquin.
\newblock Approximate fictitious play for mean field games.
\newblock {\em arXiv preprint arXiv:1907.02633}, 2019.

\bibitem{feron2020price}
Olivier F{\'e}ron, Peter Tankov, and Laura Tinsi.
\newblock Price formation and optimal trading in intraday electricity markets.
\newblock {\em arXiv preprint arXiv:2009.04786}, 2020.

\bibitem{frank1956algorithm}
Marguerite Frank and Philip Wolfe.
\newblock An algorithm for quadratic programming.
\newblock {\em Naval research logistics quarterly}, 3(1-2):95--110, 1956.

\bibitem{fudenberg1998theory}
Drew Fudenberg and David~K Levine.
\newblock {\em The theory of learning in games}, volume~2.
\newblock MIT press, 1998.

\bibitem{geist2021concave}
Matthieu Geist, Julien P{\'e}rolat, Mathieu Lauri{\`e}re, Romuald Elie, Sarah
  Perrin, Olivier Bachem, R{\'e}mi Munos, and Olivier Pietquin.
\newblock Concave utility reinforcement learning: the mean-field game
  viewpoint.
\newblock {\em arXiv preprint arXiv:2106.03787}, 2021.

\bibitem{gilbarg2015elliptic}
David Gilbarg and Neil~S. Trudinger.
\newblock {\em Elliptic partial differential equations of second order}.
\newblock springer, 2015.

\bibitem{graber2015existence}
P.~Jameson Graber and Alain Bensoussan.
\newblock Existence and uniqueness of solutions for {B}ertrand and {C}ournot
  mean field games.
\newblock {\em Applied Mathematics \& Optimization}, pages 1--25, 2015.

\bibitem{graber2020nonlocal}
P.~Jameson Graber, Vincenzo Ignazio, and Ariel Neufeld.
\newblock Nonlocal bertrand and cournot mean field games with general nonlinear
  demand schedule.
\newblock {\em Journal de Math{\'e}matiques Pures et Appliqu{\'e}es},
  148:150--198, 2021.

\bibitem{graber2018variational}
P.~Jameson Graber and Charafeddine Mouzouni.
\newblock Variational mean field games for market competition.
\newblock In {\em PDE models for multi-agent phenomena}, pages 93--114.
  Springer, 2018.

\bibitem{graber2020mean}
P.~Jameson Graber and Charafeddine Mouzouni.
\newblock On mean field games models for exhaustible commodities trade.
\newblock {\em ESAIM: Control, Optimisation and Calculus of Variations}, 26:11,
  2020.

\bibitem{graber2020weak}
P~Jameson Graber, Alan Mullenix, and Laurent Pfeiffer.
\newblock Weak solutions for potential mean field games of controls.
\newblock {\em Nonlinear Differential Equations and Applications NoDEA},
  28(5):1--34, 2021.

\bibitem{gueant2011mean}
Olivier Gu{\'e}ant, Jean-Michel Lasry, and Pierre-Louis Lions.
\newblock Mean field games and applications.
\newblock In {\em Paris-Princeton lectures on mathematical finance 2010}, pages
  205--266. Springer, 2011.

\bibitem{hadikhanloo2017learning}
Saeed Hadikhanloo.
\newblock Learning in anonymous nonatomic games with applications to
  first-order mean field games.
\newblock {\em arXiv preprint arXiv:1704.00378}, 2017.

\bibitem{HADIKHANLOO2019369}
Saeed Hadikhanloo and Francisco~J. Silva.
\newblock Finite mean field games: Fictitious play and convergence to a first
  order continuous mean field game.
\newblock {\em Journal de Mathématiques Pures et Appliquées}, 132:369 -- 397,
  2019.

\bibitem{HCMieeeAC06}
Minyi Huang, Roland~P. Malham{\'e}, and Peter~E. Caines.
\newblock Large population stochastic dynamic games: closed-loop
  {M}c{K}ean-{V}lasov systems and the {N}ash certainty equivalence principle.
\newblock {\em Communications in Information \& Systems}, 6(3):221--252, 2006.

\bibitem{jaggi2013revisiting}
Martin Jaggi.
\newblock Revisiting {F}rank-{W}olfe: Projection-free sparse convex
  optimization.
\newblock In {\em International Conference on Machine Learning}, pages
  427--435. PMLR, 2013.

\bibitem{kerdreux2020accelerating}
Thomas Kerdreux.
\newblock {\em Accelerating conditional gradient methods}.
\newblock PhD thesis, Universit{\'e} Paris sciences et lettres, 2020.

\bibitem{lachapelle2016efficiency}
Aim{\'e} Lachapelle, Jean-Michel Lasry, Charles-Albert Lehalle, and
  Pierre-Louis Lions.
\newblock Efficiency of the price formation process in presence of high
  frequency participants: a mean field game analysis.
\newblock {\em Mathematics and Financial Economics}, 10(3):223--262, 2016.

\bibitem{lachapelle2011mean}
Aim{\'e} Lachapelle and Marie-Therese Wolfram.
\newblock On a mean field game approach modeling congestion and aversion in
  pedestrian crowds.
\newblock {\em Transportation research part B: methodological},
  45(10):1572--1589, 2011.

\bibitem{LSU}
Olga~Aleksandrovna Ladyzhenskaia, Vsevolod~Alekseevich Solonnikov, and Nina~N
  Ural'tseva.
\newblock {\em Linear and quasi-linear equations of parabolic type}, volume~23.
\newblock American Mathematical Soc., 1988.

\bibitem{LL06cr1}
Jean-Michel Lasry and Pierre-Louis Lions.
\newblock Jeux {\`a} champ moyen. i--le cas stationnaire.
\newblock {\em Comptes Rendus Math{\'e}matique}, 343(9):619--625, 2006.

\bibitem{LL06cr2}
Jean-Michel Lasry and Pierre-Louis Lions.
\newblock Jeux {\`a} champ moyen. ii--horizon fini et contr{\^o}le optimal.
\newblock {\em Comptes Rendus Math{\'e}matique}, 343(10):679--684, 2006.

\bibitem{LL07mf}
Jean-Michel Lasry and Pierre-Louis Lions.
\newblock Mean field games.
\newblock {\em Japanese journal of mathematics}, 2(1):229--260, 2007.

\bibitem{lions1971optimal}
Jacques-Louis Lions.
\newblock {\em Optimal control of systems governed by partial differential
  equations}, volume 170.
\newblock Springer Verlag, 1971.

\bibitem{meszaros2015variational}
Alp{\'a}r~Rich{\'a}rd M{\'e}sz{\'a}ros and Francisco~J. Silva.
\newblock A variational approach to second order mean field games with density
  constraints: the stationary case.
\newblock {\em Journal de Math{\'e}matiques Pures et Appliqu{\'e}es},
  104(6):1135--1159, 2015.

\bibitem{monderer1996potential}
Dov Monderer and Lloyd~S Shapley.
\newblock Potential games.
\newblock {\em Games and economic behavior}, 14(1):124--143, 1996.

\bibitem{perolat2021scaling}
Julien Perolat, Sarah Perrin, Romuald Elie, Mathieu Lauri{\`e}re, Georgios
  Piliouras, Matthieu Geist, Karl Tuyls, and Olivier Pietquin.
\newblock Scaling up mean field games with online mirror descent.
\newblock {\em arXiv preprint arXiv:2103.00623}, 2021.

\bibitem{perrin2021mean}
Sarah Perrin, Mathieu Lauri{\`e}re, Julien P{\'e}rolat, Matthieu Geist, Romuald
  {\'E}lie, and Olivier Pietquin.
\newblock Mean field games flock! the reinforcement learning way.
\newblock {\em arXiv preprint arXiv:2105.07933}, 2021.

\bibitem{perrin2020fictitious}
Sarah Perrin, Julien P{\'e}rolat, Mathieu Lauri{\`e}re, Matthieu Geist, Romuald
  Elie, and Olivier Pietquin.
\newblock Fictitious play for mean field games: Continuous time analysis and
  applications.
\newblock {\em arXiv preprint arXiv:2007.03458}, 2020.

\bibitem{pieper2019linear}
Konstantin Pieper and Daniel Walter.
\newblock Linear convergence of accelerated conditional gradient algorithms in
  spaces of measures.
\newblock {\em arXiv preprint arXiv:1904.09218}, 2019.

\bibitem{prosinski2017global}
Adam Prosinski and Filippo Santambrogio.
\newblock Global-in-time regularity via duality for congestion-penalized mean
  field games.
\newblock {\em Stochastics}, 89(6-7):923--942, 2017.

\bibitem{rakotomamonjy2015generalized}
Alain Rakotomamonjy, R{\'e}mi Flamary, and Nicolas Courty.
\newblock Generalized conditional gradient: analysis of convergence and
  applications.
\newblock {\em arXiv preprint arXiv:1510.06567}, 2015.

\bibitem{robinson1951iterative}
Julia Robinson.
\newblock An iterative method of solving a game.
\newblock {\em Annals of mathematics}, pages 296--301, 1951.

\bibitem{xu2017convergence}
Hong-Kun Xu.
\newblock Convergence analysis of the {F}rank-{W}olfe algorithm and its
  generalization in {B}anach spaces.
\newblock {\em arXiv preprint arXiv:1710.07367}, 2017.

\end{thebibliography}

\end{document}